\documentclass[nospthms,natbib,graybox]{svmult}

\setcitestyle{numbers,square}
\usepackage[latin2]{inputenc}
\usepackage{amsfonts}
\let\I\undefined
\let\D\undefined
\usepackage{amssymb,amsmath,amsthm}%
\usepackage[mathcal]{euscript}

\makeatletter

\def\addtoend#1#2{%
  \expandafter\addtoend@#1{#2}{#1}}

\def\addtoend@#1\eqlist#2#3{\def#3{#1{#2}\eqlist}}

\def\addtobegin#1#2{% 
  \expandafter\addtobegin@#1\endaddtobegin@{#2}{#1}}

\def\addtobegin@\bqlist#1\endaddtobegin@#2#3{\def#3{\bqlist{#2}#1}}

\def\addtolist@#1\endaddtolist@#2{\addtoend#2{#1}}

\def\addtocar#1#2{%
  \lisp@save\aux@
  \getfirst#1\aux@
  \expandafter\addtocar@\aux@#2\endaddtocar@#1
  \lisp@restore}
\def\addtocar@#1\endaddtocar@#2{\addtobegin#2{#1}}

\def\appendto#1#2{\expandafter\appendto@#1#2#1}
\def\appendto@#1\eqlist#2#3{\expandafter\appendto@@#2\endappendto@@{#1}#3}
\def\appendto@@\bqlist#1\endappendto@@#2#3{\def#3{#2#1}}

\def\newlist#1{\def#1{\bqlist\eqlist}}

\newlist\lisp@save@list
\newcount\lisp@count
\def\lisp@savenum{0}

\def\inc#1{\lisp@count=#1\relax\advance\lisp@count by1\relax
  \edef#1{\number\lisp@count}}

\def\ginc#1{\lisp@count=#1\relax\advance\lisp@count by1\relax
  \xdef#1{\number\lisp@count}}

\def\dec#1{\lisp@count=#1\relax\advance\lisp@count by-1\relax
  \edef#1{\number\lisp@count}}

\def\add#1#2{\lisp@count=#1\relax\advance\lisp@count by #2\relax
  \edef#1{\number\lisp@count}}

\def\sub#1#2{\lisp@count=#1\relax\advance\lisp@count by -#2\relax
  \edef#1{\number\lisp@count}}

\def\lisp@save#1{%
  \inc\lisp@savenum
  \expandafter\lisp@save@\lisp@savenum\relax#1%
}

\def\lisp@save@#1\relax{%
  \expandafter\lisp@save@@\csname lisp@save@listitem#1\endcsname
}

\def\lisp@save@@#1#2{%
  \let#1=#2%
  \addtobegin\lisp@save@list{#2#1}%
}

\def\lisp@restore{%
  \mapfirst\lisp@save@list\lisp@let@
  \dec\lisp@savenum
}

\def\lisp@let@#1{\let#1}

\def\getfirst#1#2{\lisp@ifemptylist#1{\def#2{}}%
  {\expandafter\getfirst@#1\endgetfirst@#1#2}}

\def\getfirst@\bqlist#1#2\endgetfirst@#3#4{\def#3{\bqlist#2}\def#4{#1}}

\def\mapfirst#1#2{%
  \lisp@ifemptylist#1\relax{%
    \expandafter\mapfirst@#1\endmapfirst@#1#2%
  }%
}

\def\mapfirst@\bqlist#1#2\endmapfirst@#3#4{%
  \def#3{\bqlist#2}%
  #4{#1}%
}

\def\bqlist{\bqlist}
\def\eqlist{\eqlist}

\def\deflisp@ifdummy#1{%
  \expandafter\deflisp@ifdummy@\csname lisp@ifdummy#1\endcsname{#1}}

\def\deflisp@ifdummy@#1#2{%
  \expandafter\deflisp@ifdummy@@\csname #2\endcsname#1}

\def\deflisp@ifdummy@@#1#2{\def#2{%%\def\lisp@dummy{##1}%
    \ifx\lisp@dummy#1\expandafter\@firstoftwo
    \else\expandafter\@secondoftwo\fi}}

\deflisp@ifdummy{eqlist}
\deflisp@ifdummy{elt@}
\deflisp@ifdummy{empty}

\def\lisp@ifempty#1{\let\lisp@dummy#1\lisp@ifdummyempty}
\def\lisp@ifemptylist#1{\expandafter\list@ifempty@#1\endlisp@ifempty}
\def\list@ifempty@\bqlist#1\endlisp@ifempty{\def\lisp@dummy{#1}%
  \lisp@ifdummyeqlist}

\def\@ifnum#1{%
  \ifnum#1\relax\expandafter\@firstoftwo\else\expandafter\@secondoftwo\fi
}

\@ifundefined{texexp}{\let\texexp\exp\def\exp{\texexp\isparam\cbr}}\relax

\def\isparam#1{\isparamm#1\relax}

\def\isparamm{\@ifnextchar{\bgroup}}

\def\delimht#1{{%
    \let\\\relax
    \let\underbrace\dummyub
    \let\overbrace\dummyob
    \vphantom{\replace@tab#1&\endreplace@tab}}}%

\long\def\replace@tab#1&#2\endreplace@tab{%
  \@ifempty{#2}{#1}{\replace@tab#1#2\endreplace@tab}}

\long\def\replace@vline#1|#2\endreplace@vline{%
  \@ifempty{#2}{#1}{\replace@vline#1\felt#2\endreplace@vline}}

\let\felt=|
\let\endreplace@vline=\relax
\let\replace@@vline=\relax
\def\replace@@@vline#1\endreplace@vline{\replace@vline#1|\endreplace@vline}

\def\delimlr#1#2#3{%%
  \lisp@save\felt
  \def\felt{\left|\delimht{#2}\right.}%
  \left#1\delimht{#2}\hskip-2\nulldelimiterspace\relax\right.\relax%
  \replace@@vline\relax#2\endreplace@vline
  \left.\hskip-2\nulldelimiterspace\relax\delimht{#2}\right#3%
  \lisp@restore}

\def\newdelim#1#2#3{\def#1##1{\@ifnextchar{*}{\@tempswatrue}{\@tempswafalse}\delimlr#2{##1}#3}}

\def\newdelim#1#2#3{\def#1{\@ifnextchar{*}{\delimlr@star#2#3}{\delimlr@#2#3}}}
\def\delimlr@#1#2#3{\delimlr#1{#3}#2}
\def\delimlr@star#1#2*#3{#1#3#2}

\def\dummyub#1_#2{#1}
\def\dummyob#1^#2{#1}

\def\ito/{It\^o}

\def\PEzjel#1{%
  \lisp@save\replace@@vline
  \let\replace@@vline=\replace@@@vline
  \delimlr{(}{#1}{)}%
  \lisp@restore}

\newdelim\abs||
\newdelim\q\langle\rangle
\newdelim\norm\|\|
\newdelim\cbr\{\}
\newdelim\bra[]
\newdelim\zjel()

\let\event=\cbr

\let\smallset=\cbr
\def\set@internal#1#2#3{#1{#2\,:\,#3}}
\def\set{\@ifstar{\set@internal{\cbr*}}{\set@internal\cbr}}

\let\sign=\sgn

\let\br\bra
\def\eps{\varepsilon}
\def\io{\text{ i.o.}}
\usepackage{bbold}
\def\Isymb{{\mathbb 1}}

\newcommand{\I}[1][]{\Isymb
  \@ifempty{#1}{\isparamm\subscr\subnozjel}{\subnozjel{#1}}} 
\def\subscr#1{_{(#1)}}
\def\subnozjel#1{_{#1}}

\def\real{{\mathbb R}}

\def\Z{{\mathbb Z}}
\def\B{{\mathcal B}}
\def\F{{\mathcal F}}
\def\G{{\mathcal G}}

\def\cC{{\mathcal C}}
\def\dtilde#1{ \breve{#1}}%\tilde\tilde{#1}}
\def\zfrac#1#2{\zjel{\frac{#1}{#2}}}
\def\lT{\mathbf{T}}
\def\PEfont{}
\newcommand{\PE}[1][]{\PEfont{\Pe}_{#1}%
  \@ifnextchar^{\xPE}{\isparam\PEzjel}}
\def\xPE^#1{^{#1}\isparam\PEzjel}
\renewcommand{\E}{\def\Pe{E}\PE}
\renewcommand{\P}{\def\Pe{P}\PE}

\newcommand{\D}{\def\Pe{D}\PE}

\def\toby#1#2{\,{\buildrel #2\over#1}\,}

\newcommand\dto[1][\to]{\toby{#1}{d}}
\newcommand\pto[1][\to]{\toby{#1}{p}}
\newcommand\wto[1][\to]{\toby{#1}{w}}

\def\tr{\operatorname{Tr}\isparam\zjel}

\def\iff{\Leftrightarrow}

\def\@bs{\relax}
\def\stripbs#1#2\relax#3{%
  \def\next{#1}%
  \ifx\@bs\next\relax\expandafter\@firstoftwo\else\expandafter\@secondoftwo\fi
  {\def#3{#2}}{\def#3{#1#2}}%
}

\def\setbs#1#2\relax{\def\@bs{#1}}
\expandafter\stripbs\string\ \relax\@bs
\expandafter\setbs\@bs\relax\relax

\def\defname#1#2#3#4{%
  \lisp@save\nn
  \expandafter\stripbs\string#1\relax\nn
  \edef\nn{\csname#3\nn#4\endcsname}%
  \expandafter\def\nn#2\relax%
  \lisp@restore
}

\def\deftx#1{%
  \defname{#1}{{\tilde{#1}}}{t}{}%
}

\deftx\gamma

\newcommand\defxn[3][]{%
  \expandafter\newcommand\csname#2n\endcsname[1][n]{#3^{(##1)}#1}}

\def\argtosubscript{\@ifnextchar{(}{\tosub@scr}{\relax}}
\def\tosub@scr(#1){_{#1}}

\makeatother

\usepackage{xcolor}
\usepackage[colorlinks,citecolor=blue]{hyperref}

\def\levy/{L\'evy}
\newtheorem{theorem}{Theorem}
\newtheorem{lemma}[theorem]{Lemma}
\newtheorem{corollary}[theorem]{Corollary}
\newtheorem{proposition}[theorem]{Proposition}

\theoremstyle{definition}
\newtheorem{df}[theorem]{Definition}

\theoremstyle{remark}

\newtheorem*{remark*}{Remark}
\newtheorem*{remarks*}{Remarks}

\newcounter{enuma}\newcounter{enumn}\newcounter{enumr}
\newenvironment{rlist}[1][0]{%
  \begin{list}{%
      \hbox to0pt{\hss(\theenumr)}%
      \unskip\ignorespaces
    }{\usecounter{enumr}\setcounter{enumr}{#1}}%
    \def\theenumr{\roman{enumr}}%
  }%
  {%
  \end{list}%
}

\newenvironment{alist}[1][]{%
  \begin{list}{%
      \hbox to0pt{\hss(\theenuma)}%
      \unskip\ignorespaces}{\usecounter{enuma}}\def\theenuma{\alph{enuma}#1}}{\end{list}}

\newenvironment{nlist}[1][]{\begin{list}{\hbox to
      0pt{\hss(\theenumn)}\unskip\ignorespaces}{\usecounter{enumn}}\def\theenumn{\arabic{enumn}#1}}{\end{list}}

\let\PEfont\mathbf

\defxn[\argtosubscript]hh

\defxn b\beta 
\defxn BB 
\defxn XX
\defxn MM 
\def\bt{{\underline{t}}}

\makeatletter
\newcommand{\rct}[1][c]{\lisp@save\urct\def\urct{#1}\rctt}
\newcommand{\rctt}[1][t]{\rho^{\urct}_{#1}\lisp@restore}
\makeatother
\def\d{\mathrm{d}}

\begin{document}

\title*{Some sufficient conditions for the ergodicity of the
  \levy/ transformation}
\titlerunning{Ergodicity of the \levy/ transform}
\author{Vilmos Prokaj\thanks{The European Union and the European Social
    Fund have provided financial support to the project under the
    grant agreement no. T\'AMOP 4.2.1./B-09/1/KMR-2010-0003.}}
\institute{Vilmos Prokaj \at
  E\"otv\"os Lor\'and University, Department of Probability and Statistics,\\
  P\'azm\'any P. s\'et\'any 1/C, Budapest, H-1117 Hungary\\ 
  \email{prokaj@cs.elte.hu}} 
\maketitle

\abstract{We propose a possible way of attacking the question posed
  originally by Daniel Revuz and Marc Yor in their book published
  in 1991.\nocite{revuz-yor} They asked whether
  the \levy/ transformation of the Wiener--space is ergodic.
  Our main results are formulated in terms of a strongly stationary sequence
  of random variables  obtained by evaluating the iterated paths at
  time one. 
  Roughly speaking, this sequence has to approach zero ``sufficiently
  fast''. For example, one of our results states that if the expected hitting
  time of small neighborhoods of the origin do not grow faster than the
  inverse of the size of these sets then the \levy/ transformation is strongly
  mixing, hence ergodic. 
}

\section{Introduction}
\label{sec:1}

We work on the canonical space for continuous processes, that is, on
the set of continuous functions $\cC[0,\infty)$ equipped with the Borel
$\sigma$--field $\B (\cC[0,\infty))$ and the Wiener measure $\P$. On this
space the canonical process $\beta_t (\omega)=\omega (t)$ is a Brownian motion
and the \levy/ transformation $\lT$, given by the formula
\begin{displaymath}
  (\lT\beta)_t=\int_0^t \sign(\beta_s) \d\beta_s,
\end{displaymath}
is almost everywhere defined and preserves the measure $\P$. 
A long standing open question is the ergodicity of this transformation.
It was probably first mentioned in written form in \citet{revuz-yor} (pp. 257).
Since then there were some work on the question, see
\citet{MR1231991,MR1308559,MR2417970,MR1318194,MR1975087}. One of the 
recent deep result of Marc Malric, see \cite{malric-2010}, is the
topological recurrence of the transformation, that is, the orbit of a typical
Brownian path meets any non empty open set almost surely.
\citet{Leuridan2009} provide an alternative presentation of the proof. 

In this paper we consider mainly the strong mixing property of the
\levy/ transformation.   Our main results are formulated in terms of a
strongly stationary sequence   of random variables  defined by evaluating the
iterated paths at  time one. Put $Z_n=\min_{0\leq k<n} \abs*{(\lT^k\beta)_1}$.
We show in Theorem \ref{prop:8} that if 
\begin{equation}
  \liminf_{n\to\infty} \frac{Z_{n+1}}{Z_n}<1,\quad\text{almost surely}, \tag{$*$}
\end{equation}
then $\lT$ is strongly mixing, hence ergodic. 

We will say that a family of real
valued variables $\set{\xi_i}{i\in I}$ is tight if the family of the
probability measures  $\set{\P\circ \xi_i^{-1}}{i\in  I}$ is tight, that is if
$\sup_{i\in I} \P{\abs{\xi_i}>K}\to0$ as $K\to\infty$.   

In Theorem \ref{thm:tightness} below, we will see
that the tightness of the family $\set{nZ_n}{n\geq1}$ 
implies ($*$), in
particular if $\E{Z_n}=O(1/n)$ then the \levy/ transformation is strongly
mixing, hence ergodic. Another way of expressing the same idea, uses the
hitting time $\nu(x)=\inf\set{n\geq0}{Z_n<x}$ of the $x$-neighborhood of zero by the sequence $((\lT^k\beta)_1)_{k\geq0}$ for $x>0$. 
In the same Theorem we will see that the tightness of
$\set{x\nu(x)}{x\in(0,1)}$ is also sufficient for ($*$). In particular, if 
$\E{\nu(x)}=O(1/x)$ as $x\to0$,  that is, the expected hitting time of small
neighborhoods of the origin  do not grow faster than the inverse of the size
of these sets, then the \levy/ transformation is strongly mixing, hence
ergodic.

It is natural to compare our result with the density theorem of Marc Malric.
We obtain that to settle the question of ergodicity one should focus on
specific open sets only, 
but for those sets deeper understanding of the  hitting time is
required.

In the next section we sketch  our argument, formulating the intermediate
steps. Most of the proofs are given in Section~\ref{sec:3}. Note, that we do
not use the topological recurrence theorem of Marc Malric, instead all of our
argument is based on his density result of the zeros of the iterated paths,
see \cite{MR1975087}. This theorem states that the set 
\begin{equation}
  \label{eq:density}
  \set{t\geq0}{\exists n,\, (\lT^n\beta)_t=0}\quad
  \text{is dense in $[0,\infty)$ almost surely}.
\end{equation}
Hence the argument given below may eventually lead to an alternative proof of
the topological recurrence theorem as well. 

\section{Results and tools}
\label{sec:2}
\subsection{Integral-type transformations}
\label{sec:2.1}
Recall, that a measure preserving transformation $T$ of a probability
space $(\Omega,\B,\P)$ is ergodic, if 
\begin{displaymath}
  \lim_{n\to\infty}\frac1n \sum_{k=0}^{n-1}\P{A\cap T^{-k}B}=\P{A}\P{B}, 
  \quad\text{for $A,B\in\B$},
\end{displaymath}
and strongly mixing provided that
\begin{displaymath}
  \lim_{n\to\infty}\P{A\cap T^{-n}B}=\P{A}\P{B},\quad
  \text{for $A,B\in\B$}.
\end{displaymath}

The next theorem, whose proof is given in subsection~\ref{sec:3.2}, uses that   
ergodicity and strong mixing can be interpreted as asymptotic
independence when the base set $\Omega$ is a Polish space.
Here the special form of the \levy/ transformation and the one--dimensional
setting are not essential, hence we will use the phrase {\em integral-type}
for the transformation of the $d$--dimensional Wiener space in the form
\begin{equation}
  \label{eq:Th}
  T\beta=\int_0^. h(s,\beta)\d\beta_s
\end{equation}
where $h$ is a progressive $d\times d$-matrix valued function. 
It is measure-preserving, that is, $T\beta$ is a $d$--dimensional Brownian
motion, if and only if $h(t,\omega)$ is an orthogonal matrix $dt\times d\P$
almost everywhere, that is, $h^Th=I_d$, where $h^T$ denotes the transpose of
$h$ and $I_d$ is the identity matrix of size $d\times d$. Recall that
$\norm{a}_{HS}=\tr{aa^T}^{1/2}$ is the Hilbert-Schmidt norm of the matrix $a$.  

\begin{theorem}\label{thm:cf}
  Let $T$ be an integral-type measure-preserving transformation of the
  $d$--dimensional Wiener--space as in \eqref{eq:Th}  and denote by $X_n(t)$
  the process 
  \begin{equation}
    \label{eq:Xn def}
    X_n(t)=\int_0^t \hn(s) \d s
    \quad\text{with}\quad
    \hn(s)=h(s,T^{n-1}\beta)\cdots h(s,T\beta) h(s,\beta).
  \end{equation}
  Then
  \begin{rlist}
  \item\label{it:1} 
    $T$ is strongly mixing if and only if $ X_n(t)\pto0$ for all $t\geq 0$.
    \smallskip
  \item\label{it:2} $T$ is ergodic if and only if 
    $\displaystyle\frac1N\sum_{n=1}^N\norm{X_n(t)}_{HS}^2 \pto0$ for all
    $t\geq0$. 
  \end{rlist}  
\end{theorem}

The two parts of  Theorem~\ref{thm:cf} can be proved along similar lines, see
Subsection~\ref{sec:3.2}.  
Note, that the Hilbert-Schmidt norm of an orthogonal transformation in
dimension $d$ is $\sqrt{d}$ hence by \eqref{eq:Xn def} we have the trivial
bound: $\norm{X_n(t)}_{HS}\leq t\sqrt{d}$. By this boundedness the convergence
in probability is equivalent to the convergence in $L^1$ in both parts
of Theorem \ref{thm:cf}.  

\subsection{\levy/ transformation}

Throughout this section $\bn=\beta\circ \lT^n$ denotes the $n^{th}$ iterated
path under the \levy/ transformation $\lT$. Then 
$\hn(t)=\prod_{k=0}^{n-1}\sign (\bn[k]_t)$.

By boundedness, the convergence of $X_n(t)$ in probability is the same as the
convergence in $L^2$. Writing out $X_n^2(t)$ we obtain that:
\begin{equation}
  \label{eq:Xn^2}
  X_n^2(t)=2\int_{0<u<v<t}\hn(u)\hn(v)\d u\d v.
\end{equation}
Combining \eqref{eq:Xn^2} and (\ref{it:1})  of Theorem~\ref{thm:cf} we obtain
that $\lT$ is strongly mixing provided that   
\begin{equation}\label{eq:hn 1}
  \E{\hn(s)\hn(t)}\to0, \quad\text{for almost all $0<s<t$.}
\end{equation}
By scaling, $\E{\hn_s\hn_t}$  depends only on the ratio  $s/t$, and the 
sufficient condition \eqref{eq:hn 1} is even simplifies to
\begin{displaymath}
  \E{\hn(s)\hn(1)}\to0, \quad\text{for almost every $s\in(0,1)$.}
\end{displaymath}
Since $\hn(s)\hn(1)$ takes values in $\smallset{-1,+1}$ we actually have to
show that $\P{\hn(s)\hn(1)=1}-\P{\hn(s)\hn(1)=-1}\to0$. It is quite natural to
prove this limiting relation by a kind of coupling.
In the present setting this means a transformation $S$ of the state space
$\cC[0,\infty)$ preserving the Wiener measure and mapping most of the event
$\event*{\hn(s)\hn(1)=1}$ to $\event*{\hn(s)\hn(1)=-1}$ for $n$ large. 

The transformation $S$ will be  the reflection of the path after a suitably
chosen stopping time $\tau$, i.e.,
\begin{displaymath}
  (S\beta)_t =2\beta_{t\wedge \tau}-\beta_t.
\end{displaymath}

\begin{proposition}\label{prop:1}
  Let $C>0$ and $s\in (0,1)$. If there exists a stopping time $\tau$ such that
  \begin{alist}
  \item\label{it:a} $s<\tau<1$  almost surely,
  \item\label{it:b} $\nu=\inf\set{n\geq 0}{\bn_{\tau}=0}$ is finite almost
    surely, 
  \item\label{it:c} $\abs*{\bn[k]_\tau}>C\sqrt{1-\tau}$ for $0\leq k<\nu$
    almost surely. 
  \end{alist}
  then 
  \begin{displaymath}
    \limsup_{n\to\infty} \abs{\E{\hn(s)\hn(1)}}\leq 
    \P{\sup_{t\in [0,1]}\abs{\beta}>C}
  \end{displaymath}
\end{proposition}

One can relax the requirement that $\tau$ is a stopping time in
Proposition~\ref{prop:1}. 

\begin{proposition}\label{prop:2}
  Assume that for any $s<1$ and $C>0$ time there  exists  a random time $\tau$ 
  with properties (\ref{it:a}), (\ref{it:b}) and (\ref{it:c}) in
  Proposition~\ref{prop:1}.   

  Then there are also a stopping times with these properties
  for any $s<1$, $C>0$.
\end{proposition}

For a given $s\in(0,1)$ and  $C>0$, to prove the existence of the random time
$\tau$ with the prescribed properties it is natural to consider all time
points not only time one. That is, for a given path $\bn[0]$ how large is the
random set of ``good time points'', which will be denoted by $A(C,s)$:
\begin{multline}\label{eq:Adef}
  A(C,s)=\set{t>0}{\text{exist $n,\gamma,$ such that $s t<\gamma<t$,}\\
    \text{$\bn_\gamma=0$ and 
      $\inf_{0\leq k<n}\abs*{\bn[k]_\gamma}>C\sqrt{t-\gamma}$}}. 
\end{multline}
Note that it may happen that $n=0$ and then the infimum 
$\inf_{0\leq k<n}\abs*{\bn[k]_\gamma}$ is infinite.

Some basic properties of $A(C,s)$ for easier reference:
\begin{alist}
\item Invariance under scaling. For $x\neq 0$, let $\Theta_x$ denote the
  scaling of the path, $(\Theta_x\omega)(t)=x^{-1}\omega(x^2t)$. Then, since
  $\lT\Theta_x=\Theta_x\lT$ clearly holds for the \levy/ transformation
  $\lT$,  we have  
  \begin{equation}\label{eq:TxA}
    t\in A(C,s)(\omega)\quad\iff\quad x^{-2}t\in A(C,s)(\Theta_x\omega)
  \end{equation}
\item Since the scaling $\Theta_x$ preserves the Wiener--measure, the previous
  point implies that $\P{t\in A(C,s)}$ does not depend on $t>0$.  
\end{alist}
Observe that $A(C,s)$ contains an open interval on the right of every zero of
$\bn$ for all $n\geq0$. Indeed, if $\gamma$ is a zero of $\bn$ for some
$n\geq0$, then by choosing the smallest $n$ such that $\bn_\gamma=0$, one gets
that $t\in A (C,s)$ for all $t>\gamma$ such that $t-\gamma$ is small enough.
Since the union of the set of zeros of the iterated paths is dense, see
\cite{MR1975087}, we have that the set of good time points is a dense open
set. Unfortunately this is not enough for our purposes; a dense open set might
be of small Lebesgue measure.  
To prove that the set of good time points is of full Lebesgue measure, 
we borrow a notion from real analysis.
\begin{df}Let $H\subset\real$ and denote by $f(x,\eps)$ the supremum of
  the lengths of the  intervals contained in $(x-\eps,x+\eps)\setminus H$. 
  Then $H$ is \emph{porous} at $x$ if
  $\limsup _{\eps\to0+}f(x,\eps)/\eps>0$. 

  A set $H$ is called porous when it is porous at each point $x\in H$.    
\end{df}
Observe that if $H$ is porous at $x$ then its lower density 
\begin{displaymath}
  \liminf_{\eps\to0+}\frac{\lambda([x-\eps,x+\eps]\cap H)}{2\eps}\leq 
  1-\limsup_{\eps\to0+}\frac{f(x,\eps)}{2\eps}<1,
\end{displaymath}
where $\lambda$ denotes the Lebesgue measure. By Lebesgue's density theorem,
see \cite[pp. 13]{MR0071727}, the density of a measurable set exists and
equals to 1 at almost every point of the set. 
Since the closure of a porous set is also porous we obtain the well known fact
that a porous set is of zero Lebesgue measure.

\begin{lemma}\label{l7}
  Let $H$ be a random closed subset of $[0,\infty)$. If 
  $H$ is scaling invariant, that is $cH$ has the same law as $H$ for all
  $c>0$,
  then  
  \begin{displaymath}
    \event{1\not\in H}\subset \event{\text{$H$ is porous at 1}}
    \quad\text{and}\quad
    \P{\event{\text{$H$ is porous at 1}}\setminus\event{1\not\in H} }=0.
  \end{displaymath}
  That is,  the events $\event{1\not\in H}$ and
  $\event{\text{$H$ is porous at 1}}$ are equal up to a null sets.

  In particular, if $H$ is porous at $1$ almost surely, then 
  $\P{1\notin H}=1$.  
\end{lemma}
\begin{proof}
  Recall that a random closed set $H$ is a random element in the space of the
  closed subset of $[0,\infty)$ ---we denote it by $\F$---, endowed with the
  smallest $\sigma$-algebra  containing  the sets 
  $C_G=\set{F\in\F}{F\cap G\neq\emptyset}$, for all open
  $G\subset[0,\infty)$. Then it is easy to see, that 
  $\set{\omega}{\text{$H(\omega)$ is porous at 1}}$ is an event and
  \begin{align*}
    \mathbf{H}&=\set{(t,\omega)\in[0,\infty)\times\Omega}{t\in H(\omega)},\\
    \mathbf{H}_p&=
    \set{(t,\omega)\in[0,\infty)\times\Omega}
    {\text{$H(\omega)$ is porous at $t$}}
  \end{align*}
  are  measurable subsets of $[0,\infty)\times\Omega$. 
  We will also use the notation 
  \begin{displaymath}
    H_p(\omega)=
    \set{t\in[0,\infty)}{(t,\omega)\in\mathbf{H}_p}=
    \set{t\in[0,\infty)}{\text{$H(\omega)$ is porous at $t$}}. 
  \end{displaymath}

  Then for each $\omega\in\Omega$ the set $H(\omega)\cap H_p(\omega)$ is a
  porous set, hence of Lebesgue measure zero; see the remark before 
  Lemma \ref{l7}.  Whence Fubini theorem yields that
  \begin{displaymath}
    (\lambda\otimes\P)(\mathbf{H}\cap \mathbf{H}_p)=\E{\lambda(H\cap H_p)}=0.
  \end{displaymath}
  Using Fubini theorem again we get
  \begin{displaymath}
    0=(\lambda\otimes\P)(\mathbf{H}\cap\mathbf{H}_p)=
    \int_0^\infty \P{t\in H\cap H_p}dt.
  \end{displaymath}
  Since $\P{t\in H\cap H_p}$ does not depend on $t$ by the scaling invariance of
  $H$ we have that $\P{1\in H\cap H_p}=0$. 
  Now $\event{1\in H\cap H_p}=\event{1\in H_p}\setminus \event{1\not\in H}$,
  so we have shown that 
  \begin{displaymath}
    \P{\event{\text{$H$ is porous at 1}}\setminus\event{1\not\in H} }=0.
  \end{displaymath}

  The first part of the claim 
  $\event{1\not\in H}\subset\event{\text{$H$ is porous at $1$}}$ is obvious,
  since $H(\omega)$ is closed and if $1\not\in H(\omega)$ then there is an
  open interval containing 1 and disjoint from $H$. 
\end{proof}
We want to apply this Lemma to $[0,\infty)\setminus A(C,s)$, the random set of
bad time points. We have seen in \eqref{eq:TxA} that the law of
$[0,\infty)\setminus A(C,s)$ has the scaling property. 
For easier reference we state explicitly the corollary of the above argument,
that is the combination of (\ref{it:1}) in
Theorem~\ref{thm:cf}, Propositions~\ref{prop:1}--\ref{prop:2} and
Lemma~\ref{l7}: 
\begin{corollary}\label{cor:8}
  If $[0,\infty)\setminus A(C,s)$ is almost surely porous at 1 for any $C>0$
  and $s\in(0,1)$ then the \levy/ transformation is strongly mixing.
\end{corollary}
The condition formulated in terms $A(C,s)$ requires that small neighborhoods
of time 1 contain sufficiently large subintervals of $A(C,s)$. Looking at only
the left and only the right neighborhoods we can obtain Theorem~\ref{prop:7}
and \ref{prop:8} below, respectively.  

To state these results we introduce the following notations, for $t>0$
\begin{itemize}
\item
  \begin{displaymath}
    \gamma_n(t)=\max\set{s\leq t}{\bn_s=0}
  \end{displaymath}
  is the last zero before $t$,
\item 
  \begin{displaymath}
    \gamma_n^*(t)=\max_{0\leq k \leq n} \gamma_k(t),
  \end{displaymath}
  the last time $s$ before $t$ such that $\bn[0],\dots,\bn[n]$ has no zero in
  $(s,t]$, 
\item 
  \begin{displaymath}
    Z_n(t)=\min_{0\leq k < n} \abs*{\bn[k]_t}.
  \end{displaymath}
\end{itemize}
When $t=1$ we omit it from the notation, that is, $\gamma_n=\gamma_n(1)$,
$\gamma^*_n=\gamma^*_n(1)$  and $Z_n=Z_n(1)$.

\begin{theorem}\label{prop:7}
  Let 
  \begin{equation}
    \label{eq:Y def}
    Y=\limsup_{n\to\infty} \frac{Z_n(\gamma^*_n)}{\sqrt{1-\gamma^*_n}}.
  \end{equation}
  Then $Y$ is a $\lT$ invariant, $\smallset{0,\infty}$ valued random variable
  and 
  \begin{rlist}
  \item either $\P{Y=0}=1$;
  \item or $0<\P{Y=0}<1$, and then $\lT$ is not  ergodic; 
  \item or $\P{Y=0}=0$, that is $Y=\infty$ almost surely, and
    $\lT$ is strongly mixing. 
  \end{rlist}
\end{theorem}

\begin{theorem}\label{prop:8}
  Let 
  \begin{equation}
    \label{eq:X def}
    X=\liminf_{n\to\infty} \frac{Z_{n+1}}{Z_n}.
  \end{equation}
  Then $X$ is a $\lT$ invariant, $\smallset{0,1}$ valued random variable and
  \begin{rlist}
  \item either $\P{X=1}=1$;
  \item or $0<\P{X=1}<1$, and then $\lT$ is not ergodic;
  \item or $\P{X=1}=0$, that is $X=0$ almost surely, and  $\lT$ is strongly
    mixing.  
  \end{rlist}
\end{theorem}

\begin{remark*}\label{rem:XY}
  In Theorem~\ref{prop:8}, the first possibility $X=1$ looks very unlikely.
  If one  is able to exclude it, then the \levy/ $\lT$ transformation is
  either strongly mixing or not ergodic and the invariant random
  variable $X$ witnesses it. 
\end{remark*}

The statements in Theorems~\ref{prop:7} and \ref{prop:8} have similar structure, and the easy parts, the invariance
of $X$ and $Y$ are proved in subsection~\ref{sec:3.5.0}, while the more
difficult parts are proved in subsection~\ref{sec:3.4} and \ref{sec:3.5},
respectively. 

We can complement Theorems~\ref{prop:7} and \ref{prop:8} with the next
statement, which shows that $X$, $Y$ and the goodness of time $1$ for all
$C>0$ and $s\in(0,1)$ are strongly connected. Its proof is  defered to
subsection~\ref{sec:3.6} since it uses the side results of the proofs of
Theorems~\ref{prop:7} and \ref{prop:8}.

\begin{theorem}\label{thm:XY}Set
  \begin{displaymath}
    A=\bigcap_{s\in(0,1)}\bigcap_{C>0} A(C,s).
  \end{displaymath}
  Then the events $\event{1\in A}$, $\event{Y=\infty}$ and $\event{X=0}$ are
  equal up to null events.
  In particular, $X=1/(1+Y)$ almost surely. 
\end{theorem}

We close this section with a sufficient condition for $X<1$ almost surely.
For $x>0$, let $\nu(x)=\inf\set*{n\geq 0}{\abs*{\bn_1}<x}$. By the next
Corollary of the density theorem of \citet{MR1975087}, recalled in
\eqref{eq:density}, $\nu(x)$ is finite almost surely for all $x>0$.   
\begin{corollary}\label{cor:11} $\inf_n \abs*{\bn}$ is identically zero
  almost surely, that is
  \begin{displaymath}
    \P{ \inf_{n\geq 0}\abs*{\bn_t}=0,\,\forall t\geq0}=1
  \end{displaymath}
\end{corollary}

Recall that a family of real valued variables $\set{\xi_i}{i\in I}$ is
tight if $\sup_{i\in I}\P{\abs{\xi_i}>K}\to0$ as $K\to\infty$.  

\begin{theorem}\label{thm:tightness}
  The tightness of the families $\set{x\nu(x)}{x\in(0,1)}$ and
  $\set{nZ_n}{n\geq1}$ are equivalent and both imply $X<1$ almost surely,
  hence also the strong mixing property of the the \levy/
  transformation.
\end{theorem}

For the sake of completeness we state the next corollary, which is just an
easy application of the Markov inequality.
\begin{corollary}\label{cor:tightness}
  If there exists an unbounded, increasing function
  $f:[0,\infty)\to[0,\infty)$  such that $\sup_{x\in(0,1)}
  \E{f(x\nu(x))}<\infty$ or $\sup_n \E{f(nZ_n)}<\infty$  
  then the \levy/ transformation is strongly mixing.

  In particular, if $\sup_{x\in(0,1)} \E{x\nu(x)}<\infty$ or
  $\sup_n\E{nZ_n}<\infty$ then the \levy/ transformation is strongly mixing.
\end{corollary}

\section{Proofs}\label{sec:3}

\subsection{General results}\label{sec:3.1}  

First, we characterize strong mixing
and ergodicity of measure-preserving transformation over a Polish
space. This will be the key to prove Theorem~\ref{thm:cf}.
Although it seems 
to be natural, the author was not able to locate it in the literature.
\begin{proposition}\label{prop:12}
  Let $(\Omega,\B,\P,T)$ be a measure-preserving system, where $\Omega$ is a
  Polish space and $\B$ is its Borel $\sigma$-field. Then 
  \begin{rlist}
  \item $T$ is strongly
    mixing if and only if $\P\circ(T^0,T^n)^{-1}\toby\to{w} \P\otimes\P$.
    \smallskip
  \item $T$ is ergodic 
    if and only if $\frac1n\sum_{k=0}^{n-1} \P\circ(T^0,T^{k})^{-1}\toby\to{w}
    \P\otimes\P$. 
  \end{rlist}
\end{proposition}
Both part of the statement follows obviously from the following common
generalization.

\begin{proposition}\label{prop:16}
  Let $\Omega$ be a Polish space and $\mu_n,\mu$ be probability measures
  on the product $(\Omega\times\Omega,\B\times\B)$, where $\B$ is a Borel
  $\sigma$--field of $\Omega$. 

  Assume that for all $n$ the marginals of
  $\mu_n$ and $\mu$ are the same, that is for $A\in\B$ we have
  $\mu_n(A\times\Omega)=\mu(A\times\Omega)$ and
  $\mu_n(\Omega\times A)=\mu(\Omega\times A)$.
  
  Then $\mu_n\wto\mu$ if and only if $\mu_n(A\times B)\to\mu(A\times B)$
  for all $A,B\in\B$.
\end{proposition}
\begin{proof}

  Assume first that $\mu_n(A\times B)\to\mu(A\times B)$ for $A,B\in\B$.
  By portmanteau theorem, see
  \citet[Theorem 2.1]{MR0233396}, it is enough to show that for closed sets
  $F\subset\Omega\times\Omega$ the limiting relation
  \begin{equation}
    \label{eq:portmanteau}
    \limsup_{n\to\infty}\mu_n(F)\leq \mu(F) 
  \end{equation}
  holds.
  To see this, consider first a compact subset $F$ of
  $\Omega\times\Omega$ and an open set $G$ such that $F\subset G$.  
  We can take a finite covering of $F$ with open
  rectangles $F\subset \cup_{i=1}^r A_i\times B_i\subset G$, where
  $A_i,B_i\subset \Omega$ are open. 
  Since the difference of rectangles can be written as finite disjoint union
  of rectangles we can write
  \begin{displaymath}
    (A_i\times B_i)\setminus \bigcup_{k<i} (A_k\times B_k)=
    \bigcup_j (A'_{i,j}\times B'_{i,j}),     
  \end{displaymath}
  where $\set{A'_{i,j}\times B'_{i,j}}{i,j}$ is a finite collection of
  disjoint rectangles. By assumption 
  \begin{displaymath}
    \lim_{n\to\infty} \mu_n\zjel{A'_{i,j}\times B'_{i,j}} =
    \mu\zjel{A'_{i,j}\times B'_{i,j}},
  \end{displaymath}
  which yields
  \begin{displaymath}
    \limsup_{n\to\infty} \mu_n(F)\leq
    \lim_{n\to\infty}\mu_n\zjel{\bigcup_{i}(A_{i}\times B_{i})}=
    \mu\zjel{\bigcup_{i}(A_{i}\times B_{i})}\leq \mu(G).
  \end{displaymath}
  Taking infimum over $G\supset F$, \eqref{eq:portmanteau} follows for compact
  sets. 

  For a  general closed $F$, let $\eps>0$ and denote by
  $\mu^1(A)=\mu(A\times\Omega)$,  
  $\mu^2(A)=\mu(\Omega \times A)$ the marginals of $\mu$.
  By the tightness of $\smallset{\mu^1,\mu^2}$, one can find a compact set $C$
  such that $\mu^1(C^c)=\mu(C^c\times\Omega)\leq \eps$ and 
  $\mu^2(C^c)=\mu(\Omega\times C^c)\leq \eps$.
  Then
  \begin{displaymath}
    \mu_n(F)\leq \mu_n(F\cap (C\times C)) +2\eps. 
  \end{displaymath}
  Since $F'=F\cap (C\times C)$ is compact, we have that
  \begin{displaymath}
    \limsup_{n\to\infty} \mu_n(F)\leq 
    \limsup_{n\to\infty}\mu_n(F') + 2\eps
    \leq \mu(F')+2\eps \leq \mu(F)+2\eps.
  \end{displaymath}
  Letting $\eps\to0$ finishes this part of the proof.
  
  \bigskip
  
  For the converse, note that  $\mu^1$ and $\mu^2$ are regular since $\Omega$
  is a Polish space and $\mu^1$, $\mu^2$ are probability measures on its Borel
  $\sigma$-field.
  
  Fix $\eps>0$. For $A_i\in\B$ one can find, using the regularity of
  $\mu^i$, closed sets  $F_i$ and open sets $G_i$ 
  such that $F_i\subset A_i\subset G_i$ and $\mu^i(G_i\setminus F_i)\leq
  \eps$. Then
  \begin{displaymath}
    (G_1\times G_2) \setminus (F_1\times F_2)\subset
    ((G_1\setminus F_1)\times\Omega) \cup (\Omega\times(G_2\setminus F_2))
  \end{displaymath}
  yields that
  \begin{align*}
    \mu_n(A_1\times A_2)\leq \mu_n(G_1&\times G_2)\leq \mu_n(F_1\times
    F_2)+2\eps,\\
    \mu_n(A_1\times A_2)\geq \mu_n(F_1&\times F_2)\geq \mu_n(G_1\times
    G_2)-2\eps,
  \end{align*}
  hence by portmanteau theorem $\mu_n\wto \mu$ gives
  \begin{align*}
    \limsup_{n\to\infty}\mu_n(A_1\times A_2)\leq
    \mu(F_1&\times F_2)+2\eps\leq \mu(A_1\times A_2) +2\eps\\
    \liminf_{n\to\infty}\mu_n(A_1\times A_2)\geq
    \mu(G_1&\times G_2)-2\eps\geq\mu(A_1\times A_2)-2\eps.
  \end{align*}
  Letting $\eps\to 0$ we get $\lim_{n\to\infty}\mu_n(A_1\times
  A_2)=\mu(A_1\times A_2)$. 
\end{proof}

\subsection{Proof of Theorem~\ref{thm:cf}}
\label{sec:3.2}

\begin{proof}[Proof of the sufficiency of the conditions in
  Theorem~\ref{thm:cf}]  
  We start with the strong mixing 
  case. We want to show that 
  \begin{equation}\label{eq:X_n def}
    X_n(t)=\int_0^t \hn(s)\d s\pto 0,\quad\text{for all $t\geq0$}, 
  \end{equation}
  where $\hn(s)$ is given by \eqref{eq:Xn def}, implies the strong mixing of
  the integral-type measure-preserving  transformation $T$.  

  Actually, we show by characteristic function method that \eqref{eq:X_n def} 
  implies that the finite dimensional marginals of $(\beta,\bn)$  converge in
  distribution to the appropriate marginals of a $2d$--dimensional Brownian
  motion. 
  Then, since the sequence $(\beta,\bn)_{n\geq 0}$ is tight, not only the
  finite dimensional marginals but the sequence of
  processes $(\beta,\bn)$ converges in distribution to a  $2d$--dimensional
  Brownian motion. By Proposition~\ref{prop:12} this is equivalent with the
  strong mixing property of $T$. 

  Let $\bt=(t_1,\dots,t_k)$ be a finite subset of $[0,\infty)$. 
  Then the characteristic
  function of
  $(\beta_{t_1},\dots,\beta_{t_k},\bn_{t_1},\dots,\bn_{t_k})$ 
  can be written as 
  \begin{equation}\label{eq:11.1/3}
    \begin{split}
      \phi_n(\alpha)&
      =\E{\exp{i\int_0^\infty f\d\beta+i\int_0^\infty g\d\bn}}
      \\ &
      =\E{\exp{i\int_0^\infty (f+g\hn)\d\beta}},  
    \end{split}
  \end{equation}
  where $f,g$ are deterministic step function obtained from the time vector
  $\bt$ and $\alpha=(a_1,\dots,a_k,b_1,\dots,b_k)$; here $a_i,b_j$ are
  $d$-dimensional row vectors and
  \begin{displaymath}
    f=\sum_{j=1}^k a_j \I[{[0,t_j]}],\quad\text{and}\quad    
    g=\sum_{j=1}^k b_j \I[{[0,t_j]}] .
  \end{displaymath}
  We have to  show that  
  \begin{displaymath}
    \phi_n(\alpha)\to \phi(\alpha)=
    \exp{-\frac12\int_0^\infty  (\abs{f}^2+\abs{g}^2)} \quad
    \text{as  $n\to\infty$.} 
  \end{displaymath}
  Using that $\bn=\int \hn \d\beta$ and
  \begin{displaymath}
    M_t=\exp{i \int_0^t (f(s)+g(s)\hn(s))\d\beta_s+
      \frac12\int_0^t \abs{f(s)+g(s)\hn(s)}^2\d s} 
  \end{displaymath}
  is a uniformly integrable martingale starting from 1, 
  we obtain that 
  $\E{M_\infty}=1$ and
  \begin{multline}
    \label{eq:11.2/3}
    \phi(\alpha)=\phi(\alpha)\E{M_\infty}=\\
    \E{\exp{i \int_0^\infty
        (f(s)+g(s)\hn(s))\d\beta_s+\int_0^\infty g(s)\hn(s)f^T(s)\d s}}
  \end{multline}
  As $\exp\cbr*{i\int_{[0\infty)}(f+g\hn)\d \beta}$ is of modulus one, we
  get from \eqref{eq:11.1/3} and \eqref{eq:11.2/3} that
  \begin{equation}
    \label{eq:11.3/3}
    \abs{\phi(\alpha)-\phi_n(\alpha)}\leq
    \E{\abs{\exp{\int_0^\infty g(s)\hn(s)f^T(s)\d s}-1}}.
  \end{equation}
  Note that $f^Tg$ is a matrix valued  function of the form
  $f^Tg=\sum_{j=1}^k c_j\I[{[0,t_j]}]$, hence
  \begin{displaymath}
    \int_0^\infty g(s)\hn(s)f^T(s)\d s=
    \int_0^\infty \tr{f^T(s)g(s)\hn(s)}\d s=\sum_{j=1}^k \tr{c_jX_n(t_j)},
  \end{displaymath}
  and $\abs*{\int_0^\infty g(s)\hn(s)f^T(s)\d s}\leq
  M=\int_0^\infty\abs{f(s)}\abs{g(s)}\d s<\infty$. 
  With this notation, 
  using $\abs{e^x-1}\leq \abs{x}e^{\abs{x}}$ for
  $x\in\real$ and $\abs{\tr(ab)}\leq \norm{a}_{HS}\norm{b}_{HS}$, 
  we can continue \eqref{eq:11.3/3} to get 
  \begin{equation}\label{eq:phi_n-phi}
    \begin{split}
      \abs{\phi_n(\alpha)-\phi(\alpha)}& \leq 
      \E{\abs{\exp{\int_0^\infty g(s)\hn(s)f^T(s)\d s}-1}}\\ & \leq
      e^M\E{\abs{\sum_{j=1}^k \tr{c_jX_n(t_j)}}} 
      \\ &\leq
      e^M\sum_{j=1}^k \norm{c_j}_{HS}\E{\norm{X_n(t_j)}_{HS}}.  
    \end{split}
  \end{equation}
  Since $\norm{X_n(t_j)}_{HS}\leq t_j\sqrt{d}$ and $X_n(t_j)\pto0$ by
  assumption, we obtained that  $\phi_n(\alpha)\to\phi(\alpha)$ and the
  statement follows. 

  To prove (\ref{it:2}) we use the same method. We introduce
  $\kappa_n$ which is a random variable independent of the sequence
  $(\bn)_{n\in\Z}$ and uniformly distributed  on
  $\smallset{0,1,\dots,n-1}$. Ergodicity can be formulated as
  $(\beta,\bn[\kappa_n])$ converges in distribution to a $2d$--dimensional
  Brownian motion. The joint characteristic function $\psi_n$ of
  $(\beta_{t_1},\dots,\beta_{t_k},\bn[\kappa_n]_{t_1},\dots,\bn[\kappa_n]_{t_k})$ 
  can be expressed, similarly as above,
  \begin{displaymath}
    \psi_n=\frac1n\sum_{\ell=0}^{n-1}\phi_\ell
  \end{displaymath}
  where $\phi_\ell$ is as in the first part of the proof.
  Using the estimation \eqref{eq:phi_n-phi} obtained in the first part
  \begin{align*}
    \abs{\phi(\alpha)-\psi_n(\alpha)}& \leq  
    \frac1n\sum_{\ell=0}^{n-1}\abs{\phi(\alpha)-\phi_\ell(\alpha)} \\ & \leq
    \frac{e^M}{n}\sum_{\ell=0}^{n-1}\sum_{j=0}^k
    \norm{c_j}_{HS}\E{\norm{X_{\ell}(t_j)}_{HS}} \\ & =
    e^M\sum_{j=1}^k\norm{c_j}_{HS}
    \E{\frac{1}{n}\sum_{\ell=0}^{n-1}{\norm{X_{\ell}(t_j)}_{HS}}}.
  \end{align*}
  Now $\abs{\phi(\alpha)-\psi_n(\alpha)}\to0$ follows from our
  condition in part
  (\ref{it:2}) by the Cauchy--Schwartz inequality, since 
  \begin{displaymath}
    \zjel{\frac1n\sum_{\ell=0}^{n-1}\norm{X_\ell(t_j)}_{HS}}^2\leq
    \frac1n\sum_{\ell=0}^{n-1}\norm{X_\ell(t_j)}_{HS}^2 
    \pto 0.
  \end{displaymath}
  and $\frac1n\sum_{\ell=0}^{n-1}\norm{X_\ell(t_j)}^2_{HS}\leq t^2_jd$.
\end{proof}

\begin{proof}[Proof of the necessity of the conditions in Theorem~\ref{thm:cf}]
  Recall that the quad\-ratic variation of an $m$--dimensional martingale
  $M=(M_1,\dots,M_m)$ is a matrix valued process whose $(j,k)$ entry is
  $\q{M_j,M_k}$.  
  The proof of the following fact can be found in \cite{MR959133}, see
  corollary 6.6 of Chapter VI. 

  \smallskip  

  { \narrower\it
    \noindent Let $(\Mn)$ be a  sequence of $m$--dimensional,  continuous
    local martingales. If $\Mn\dto M$ then  $(\Mn,\q{\Mn})\dto (M,\q{M})$.  
    \par
  }
  
  \medskip
  
  By enlarging the probability space, we may assume that there is a random
  variable $U$, which  is uniform on $(0,1)$ and independent of $\beta$. Denote
  by $\kappa_n=[nU]$ the integer part of $nU$. Let $\G$ be the smallest
  filtration satisfying the usual hypotheses, making $U$ $\G_0$ measurable and
  $\beta$ adapted to $\G$.  
  Then $\beta$ is a Brownian motion in $\G$;  $(\beta,\bn)$ and
  $(\beta,\bn[\kappa_n])$ are continuous martingales in $\G$.
  The quadratic covariations are
  \begin{displaymath}
    \q*{\bn,\beta}_t=\int_0^t \hn(s)ds=X_n(t),\quad\text{and}\quad
    \q*{\bn[\kappa_n],\beta}_t=\sum_{k=0}^{n-1}\I{\kappa_n=k}X_k(t).
  \end{displaymath}
  
  By Proposition \ref{prop:2}, the strong mixing property and the ergodicity
  of $T$ are respectively equivalent to the convergence in distribution of 
  $(\beta, \bn)$ and $(\beta, \bn[\kappa_n])$  to a $2d$--dimensional Brownian
  motion. 

  By the fact just recalled, the strong mixing property of $T$ implies that
  $\q*{\bn,\beta}_t\dto 0$, while its ergodicity ensures that  
  $\q*{\bn[\kappa_n],\beta}_t\dto0$ for every $t \geq0$.  Since the limit is
  deterministic, the convergence also holds in probability.  
  The ``only if'' part of (\ref{it:1}) follows immediately.

  For the ``only if'' part of (\ref{it:2}) we add that
  \begin{displaymath}
    \norm*{\q*{\bn[\kappa_n],\beta}_t}_{HS}^2=
    \norm{\sum_{k=0}^{n-1} \I{\kappa_n=k} X_k(t)}^2_{HS}=
    \sum_{k=0}^{n-1} \I{\kappa_n=k}\norm{X_k(t)}_{HS}^2
  \end{displaymath}
  Since $\norm*{X_k(t)}^2_{HS}\leq t^2 d$ the convergence in probability of
  $\q*{\bn[\kappa_n],\beta}_t$ to zero is also a convergence of
  $\norm{\q*{\bn[\kappa_n],\beta}_t}_{HS}^2$ to zero in $L^1(\P)$, which
  implies the convergence in $L^1(\P)$ to zero of the conditional expectation
  \begin{displaymath}
    \E{\norm*{\q*{\bn[\kappa_n],\beta}_t}_{HS}^2|\sigma(\beta)}=
    \frac1n\sum_{k=0}^{n-1} \norm{X_k(t)}_{HS}^2.
  \end{displaymath}
  The ``only if'' part of (\ref{it:2}) follows.
\end{proof}

\subsection{First results for the \levy/ transformation}
\label{sec:3.3}

We will use the following property of the \levy/ transformation many times.
Recall that $\lT^n\beta=\beta\circ \lT^n$ is also denoted by $\bn$.
We will also use the notation $\hn(t)=\prod_{k=0}^{n-1} \sign (\bn[k]_t)$
for $n\geq 1$ and $\hn[0]=1$. 

\begin{lemma}\label{l1}On an almost sure event the following property holds:
  
  \noindent For any interval $I\subset[0,\infty)$, point $a\in I$ and integer
  $n>0$,  if  
  \begin{equation}\label{l1:cond}
    \sup_{t\in I}\abs*{\beta_t-\beta_a}< \min_{0\leq k<n}\abs*{(\lT^k\beta)_a}   
  \end{equation}
  then
  \begin{rlist}
  \item\label{l1:it1} $\lT^k\beta$ has no zero in $I$,  for $0\leq k\leq n-1$,
  \item\label{l1:it2}
    $(\lT^k\beta)_t-(\lT^k\beta)_a=\hn[k] (a)\zjel{\beta_t-\beta_a}$ for
    $t\in I$ 
    and $0\leq k\leq n$.  

    In particular,
    $\abs*{(\lT^k\beta)_t-(\lT^k\beta)_a}=\abs*{\beta_t-\beta_a}$ for  $t\in I$
    and $0\leq k\leq n$. 
  \end{rlist}
\end{lemma}
\begin{proof} In the next argument we only use that if $\beta$ is a Brownian
  motion and $L$ is its local time at level zero then the points of increase
  for $L$ is exactly the zero set of $\beta$ and $\lT\beta=\abs{\beta}-L$
  almost surely. Then there is $\Omega'$ of full probability such that on
  $\Omega'$ both properties hold for $\lT^n\beta$ for all $n\geq0$
  simultaneously.

  Let $N=N(I)=\inf\set{n\geq 0}{\text{$\lT^n\beta$ has a zero in $I$}}$.
  Since $\lT$ acts as $\lT\beta=\abs{\beta}-L$, if $\beta$ has no zero in $I$
  we have
  \begin{displaymath}
    \lT\beta_t=\sign(\beta_a)\beta_t-L_a,\quad\text{for $t\in I$}.
  \end{displaymath}
  But, then $\lT\beta_t-\lT\beta_a=\sign(\beta_a)(\beta_t-\beta_a)$ and
  $\abs{\lT\beta_t-\lT\beta_a}=\abs{\beta_t-\beta_a}$  
  for $t\in I$. Iterating it we obtain that
  \begin{equation}\label{eq:n N}
    \begin{split}
      (\lT^k\beta)_t-(\lT^k\beta)_a&=\hn[k](a)\zjel{\beta_t-\beta_a},\\
      \abs{(\lT^k\beta)_t-(\lT^k\beta)_a}&=\abs{\beta_t-\beta_a},  
    \end{split}
    \quad\text{on $\event{k\leq N}$ and for $t\in I$}. 
  \end{equation}
  Now assume that \eqref{l1:cond} holds. Then, necessarily $n\leq N$ as the
  other  possibility  would lead to a contradiction. Indeed, if $N< n$ then
  $N$ is finite, $\lT^N\beta$ has a zero $t_0$ in $I$ and  
  \begin{displaymath}
    0 =
    \abs{\lT^N\beta_{t_0}}=
    \abs{\lT^N\beta_{a}}-\abs{\lT^N\beta_{t_0}-\lT^N\beta_{a}}\geq
    \min_{0\leq k< n} \abs{\lT^k\beta_a}-\sup_{t\in I}\abs{\beta_t-\beta_a}>0.
  \end{displaymath}
  So \eqref{l1:cond} implies that $n\leq N$, which proves (\ref{l1:it1}) by
  the definition of $N$ and also (\ref{l1:it2}) by \eqref{eq:n N}. 
\end{proof}

Combined with the densities of zeros, Lemma~\ref{l1} implies
Corollary~\ref{cor:11} stated above. 
\begin{proof}[Proof of Corollary~\ref{cor:11}]The statement here is that
  $\inf_{n\geq 0} \abs*{(\lT^n\beta)_t}=0$ for all $t\geq0$. 

  Assume  that
  for $\omega\in\Omega$ there is some $t>0$, such that  $\inf_{n\geq 0}
  \abs*{(\lT^n\beta)_t}$ is not zero at $\omega$.
  Then there is a neighborhood $I$ of $t$ such that 
  \begin{displaymath}
    \sup_{s\in I}\abs{\beta_s-\beta_t}<\inf_k\abs{(\lT^k\beta)_t}.  
  \end{displaymath}
  Using Lemma~\ref{l1}, we would get that for this $\omega$ the iterated paths
  $\lT^k\beta(\omega)$, $k\geq 0$ has no zero in $I$. However, 
  since 
  \begin{displaymath}
    \set{t\geq 0}{\exists k,\, (\lT^k\beta)_t=0}
  \end{displaymath}
  is dense in $[0,\infty)$ almost surely by the result of \citet{MR1975087},
  $\omega$ belongs to the exceptional negligible set.
\end{proof}

\begin{proof}[Proof of Proposition~\ref{prop:1}]
  Let $C>0$ and $s\in(0,1)$ as in the statement and assume that $\tau$ is a
  stopping time satisfying (\ref{it:a})-(\ref{it:c}), that is, $s<\tau<1$, and
  for the almost surely finite random index $\nu$ we have $\bn[\nu]_\tau=0$
  and $\min_{0\leq k<\nu} \abs*{\bn[k]_\tau}>C\sqrt{1-\tau}$. Recall that $S$
  denotes the reflection of the trajectories after $\tau$.

  Set $\eps_n=\hn(s)\hn(1)$ for $n>0$  and 
  \begin{displaymath}
    A_{C}=\event{\sup_{t\in[\tau,1]}\abs*{\bn[0]_t-\bn[0]_\tau}\leq
      C\sqrt{1-\tau}}.  
  \end{displaymath}
  We show below that on the event $A_C\cap\event{n>\nu}$, we have
  $\eps_n=-\eps_n\circ S$. Since $S$ preserves the Wiener measure 
  $\P$, this implies that  
  \begin{align*}
    \abs{\E{\eps_n}}=\frac12 \abs{\E{\eps_n+\eps_n\circ S}}
    &\leq \frac12\E{\abs{\eps_n+\eps_n\circ S}}\\
    &=\P{\eps_n=\eps_n\circ S}\\ 
    &\leq \P{A_C^c\cup \event{n\leq \nu}}\leq
    \P{A_C^c}+\P{n\leq \nu}
  \end{align*}
  When $n\to\infty$, this yields
  \begin{displaymath}
    \limsup_{n\to\infty} \abs{\E{\hn(s)\hn(1)}}\leq \P{A_C^c}=
    \P{\sup_{s\in[0,1]}\abs{\beta_s}>C},
  \end{displaymath}
  by the Markov property and the scaling property of the Brownian motion.

  It remains to show that on $A_{C}\cap\event{n>\nu}$ the identity 
  $\eps_n=-\eps_n\circ S$ holds. By definition of $S$, the trajectory of
  $\beta$ and $\beta\circ S$ coincide on $[0,\tau]$, hence $\hn[k]$ and
  $\hn[k]\circ S$ coincide on $[0,\tau]$ for $k> 0$. In particular,
  $\hn[k](\tau)=\hn[k](\tau)\circ S$ and $\hn[k](s)=\hn[k](s)\circ S$ for all
  $k$ since $\tau>s$.

  On the event $A_C$ we can apply Lemma~\ref{l1} with $I=[\tau,1]$, $a=\tau$
  and $n=\nu$ to both $\beta$ and $S\circ\beta$ to get that
  \begin{equation}
    \label{eq:tau,1}
    \begin{split}
      \bn[k]_t-\bn[k]_\tau&=\hn[k](\tau)(\beta_t-\beta_\tau),\\
      \bn[k]_t\circ S-\bn[k]_\tau\circ S&=-\hn[k](\tau)(\beta_t-\beta_\tau),
    \end{split}
    \quad k\leq \nu,\,t\in[\tau,1].
  \end{equation}
  We have used that $\hn[k]_\tau=\hn[k]_\tau\circ S $ and $\beta_t\circ
  S_t-\beta_\tau\circ S=-(\beta_t-\beta_\tau)$ for $t\geq\tau$ by the
  definition of $S$.

  Using that on $A_C$ 
  \begin{displaymath}
    \abs*{\bn[k]_\tau}> C\sqrt{1-\tau}\geq \abs{\beta_1-\beta_\tau},
    \quad\text{for $k<\nu$}
  \end{displaymath}
  we get immediately from \eqref{eq:tau,1} that
  $\sign(\bn[k]_1)=\sign(\bn[k]_1)\circ S $ for $k<\nu$.

  Since $\bn[\nu]_\tau=(\bn[\nu]_\tau)\circ S=0$, for $k=\nu$
  \eqref{eq:tau,1} gives that $\bn[\nu]$ and $\bn[\nu]\circ S$ coincide on
  $[0,\tau]$ and are opposite of each other on $[\tau,1]$. 
  Hence, $\bn[k]$ and $\bn[k]\circ S$ coincide on
  $[0,1]$ for  every $k>\nu$.
  
  As a result on the event $A_C$,
  \begin{displaymath}
    \sign(\bn[k]_1)\circ S =
    \begin{cases}
      \sign(\bn[k]_1),&\text{if $k\neq\nu$},\\
      -\sign(\bn[k]_1), &\text{if $k=\nu$}  
    \end{cases}
  \end{displaymath}
  hence $\hn(1)\circ S=-\hn(1)$  on $A_C\cap\event{n>\nu}$. Since 
  $\hn(s)\circ S =\hn(s)$ for all $n$ we are done.
\end{proof}

\begin{proof}[Proof of Proposition~\ref{prop:2}]
  Let $C>0$ and $s\in(0,1)$. Call $\tau$ the infimum of those time points that
  satisfy (\ref{it:b}) and (\ref{it:c}) of Proposition~\ref{prop:1}
  with $C$ replaced by $2C$, namely $\tau=\inf_n\tau_n$, where
  \begin{align*}
    \tau_n&=\inf\set{t>s}{\bn_t=0,\,\forall k<n,\,
      \abs*{\bn[k]_t}> 2C\sqrt{(1-t)\vee0}}. 
  \end{align*}

  By assumption $\tau_n<1$ for some $n\geq 0$. Furthermore, there exists some
  finite index $\nu$ such that $\tau=\tau_\nu$. Otherwise, there would exist
  a subsequence $(\tau_n)_{n\in D}$ bounded by 1 and converging
  to $\tau$. For every $k$ one has $k<n$  for infinitely many $n\in D$,
  hence $\abs*{\bn[k]_{\tau_n}}\geq 2C\sqrt{1-\tau_n}$ by the choice of $D$. 
  Letting $n\to\infty$
  yields $\abs{\bn[k]_\tau}\geq2C\sqrt{1-\tau}>0$ for every $k$. This can
  happen only with probability zero by Corollary~\ref{cor:11}. 

  As $\nu$ is almost surely finite and $\tau=\tau_\nu$ we get that
  $\bn[\nu]_{\tau}=0$ and 
  \begin{displaymath}
    \inf\set*{\abs*{\bn[k]_\tau}}{k<\nu}\geq
    2C\sqrt{1-\tau}>C\sqrt{1-\tau}.
  \end{displaymath}
  We have that $\tau>s$ holds almost surely, since $s$ is not a zero
  of any $\bn$ almost surely, so $\tau$ satisfies (\ref{it:a})-(\ref{it:c}) of
  Lemma~\ref{prop:1}. 
\end{proof}

\subsection{Easy steps of the proof of Theorem~\ref{prop:7} and \ref{prop:8}}
\label{sec:3.5.0}

The main step of the proof of these theorems, that will be given in
subsection~\ref{sec:3.5} and \ref{sec:3.6}, is that if $Y>0$ almost surely
(or $X<1$ almost 
surely), then for any $C>0$, $s\in(0,1)$ the set of the bad time points 
$[0,\infty)\setminus A(C,s)$ is almost surely porous at 1.
Then Corollary~\ref{cor:8} applies and  the \levy/ transformation
$\lT$ is strongly mixing. 

If $Y>0$ does not hold almost surely, then either $Y=0$ or $Y$ is a
non-constant variable invariant for $\lT$, hence in latter case the \levy/
transformation $\lT$ is not ergodic. These are the first two cases in
Theorem~\ref{prop:7}. Similar analysis applies to $X$ and
Theorem~\ref{prop:8}.   

To show the invariance of $Y$ recall that $\gamma_n^*\to 1$ by  the 
density theorem of the zeros due to  \citet{MR1975087} and $\gamma_0<1$, both
property holding almost surely. 
Hence, for every large enough $n$, $\gamma_{n+1}^* >\gamma_0$, therefore 
$\gamma_{n+1}^*=\gamma_n^*\circ\lT$, 
\begin{displaymath}
  Z_n(\gamma^*_n)\circ \lT = \min_{0\leq k<n} \abs*{\bn[k+1]_{\gamma_n^*\circ
      \lT}}=
  \min_{1\leq k<n+1} \abs*{\bn[k]_{\gamma^*_{n+1}}}\geq
  Z_{n+1}(\gamma^*_{n+1}),
\end{displaymath}
and
\begin{displaymath}
  \frac{Z_n(\gamma^*_n)}{\sqrt{1-\gamma^*_n}}\circ \lT \geq
  \frac{Z_{n+1}(\gamma^*_{n+1})}{\sqrt{1-\gamma^*_{n+1}}}. 
\end{displaymath}
Taking limit superior we obtain that $Y\circ \lT \geq Y$. 
Using that $\lT$ is measure--preserving we conclude 
$Y\circ \lT=Y$ almost surely, that is, $Y$ is  $\lT$ invariant. 

To show the invariance of $X$ directly, without referring to Theorem
\ref{thm:XY},  we use Corollary \ref{cor:11}, which says that 
almost surely $\inf_{n\geq 0}\abs*{\bn_t}=0$ for all $t\geq
0$.  Thus  $Z_n\to 0$ and since $\abs*{\bn[0]_1}>0$ almost surely, for every
large enough $n$,  $Z_n<\abs*{\bn[0]_1}$, therefore
$(Z_{n+1}/Z_n)\circ\lT=(Z_{n+2}/Z_{n+1})$. Hence  $X\circ \lT=X$.

\subsection{Proof of Theorem~\ref{prop:7}}
\label{sec:3.4}

Fix $C>0$ and  $s\in(0,1)$ and consider the random set
\begin{multline}\label{tilde ACs} 
  \tilde A(C,s)=
  \set{t>0}{\text{exist $n\geq1$ such that  $s t<\gamma_n(t)=\gamma_n^*(t)$ and
    }\\ 
    \text{$\min_{0\leq k<n}\abs*{\bn[k]_{\gamma_n(t)}}>C\sqrt{t-\gamma_n(t)}$}
  }
  \subset A(C,s).
\end{multline}
The difference between $A(C,s)$ and $\tilde A(C,s)$ is that in the latter
case we only consider last zeros satisfying $\gamma_n(t)>\gamma_k(t)$ for
$k=0,\dots,n-1$, whereas in the case of $A(C,s)$ we consider any zero of
the iterated paths. Note also, that here $n>0$, so the zeros of $\beta$
itself are not used, while $n$ can be zero in the definition of $A(C,s)$.

We prove below the next proposition.
\begin{proposition}\label{prop:7new}
  Almost surely on the event $\event{Y>0}$, the closed set
  $[0,\infty)\setminus\tilde{A}(C,s)$ is porous at $1$ for any $C>0$ and
  $s\in(0,1)$. 
\end{proposition}

This result readily implies that if $Y>0$ almost surely, then 
$[0,\infty)\setminus\tilde{A}(C,s)$ and the smaller
random closed set 
$[0,\infty)\setminus A(C,s)$ are both almost surely porous at 1 for any $C>0$
and $s\in(0,1)$. Then the strong mixing property of $\lT$ follows by
Corollary~\ref{cor:8}.  

It remains to show that $Y=\infty$ almost surely on the event $\event{Y>0}$,
which proves that $Y\in\smallset{0,\infty}$ almost surely. 
This is the content of the next Proposition.

\begin{proposition}\label{prop:Y>0}Set
  \begin{displaymath}
    \tilde{A}(s)=\bigcap _{C>0} \tilde{A}(C,s),
    \quad\text{for $s\in(0,1)$ and}\quad
    \tilde{A}=\bigcap _{s\in(0,1)} \tilde{A}(s).
  \end{displaymath}
  Then the events $\event{Y>0}$, $\event{Y=\infty}$,
  $\event*{1\in\tilde{A}(s)}$, $s\in(0,1)$ and $\event*{1\in \tilde{A}}$ are
  equal up to null sets. 
\end{proposition}

\begin{proof}[Proof of Proposition~\ref{prop:Y>0}]
  Recall that $Y=\limsup_{n\to\infty} Y_n$ with
  \begin{displaymath}
    Y_n=\frac{\min_{0\leq k < n}\abs*{\bn[k]_{\gamma_n^*}}}{\sqrt{1-\gamma_n^*}}.
  \end{displaymath}
  With this notation, on $\event*{1\in \tilde{A}(C,s)}$ there is a random $n\geq 1$
  such that $Y_n>C$. Here, the restriction $n\geq1$  
  in the definition of
  $\tilde{A}(C,s)$ is useful. 
  This way, we get that $\sup_{n\geq 1} Y_n\geq C$ on $\event*{1\in
    \tilde{A}(C,s)}$ and $\sup_{n\geq1}Y_n=\infty$ on
  $\event*{1\in \tilde{A}(s)}$. 
  Since $Y_n<\infty$  almost surely for all $n\geq 1$, we
  also have that $Y=\infty$ almost surely on   $\event*{1\in \tilde{A}(s)}$.  

  Next, the law of the random closed set $[0,\infty)\setminus\tilde{A}(C,s)$
  is invariant by scaling, hence by Proposition \ref{prop:7new} and Lemma
  \ref{l7}, 
  \begin{displaymath}
    \event{Y>0}\subset\event{\text{$[0,\infty)\setminus\tilde{A}(C,s)$ is
        porous at $1$}}\subset \event{1\in\tilde{A}(C,s)},
    \quad\text{almost surely}.
  \end{displaymath}
  The inclusions $\tilde{A}(C,s)\subset\tilde{A}(C',s)$ for $C>C'$ and
  $\tilde{A}(C,s)\subset\tilde{A}(C,s')$ for $1>s'>s>0$ yield
  \begin{displaymath}
    \tilde{A}=\bigcap_{k=1}^{\infty} \tilde{A}(k,1-1/k).
  \end{displaymath}
  Thus, $\event{Y>0}\subset\event*{1\in\tilde{A}}$ almost surely.

  Hence, up to null events,
  \begin{displaymath}
    \event{Y>0}\subset\event*{1\in\tilde{A}}\subset
    \event*{1\in \tilde{A}(s)}\subset
    \event{Y=\infty}\subset\event{Y>0}
  \end{displaymath}
  for any $s\in(0,1)$, which completes the proof.
\end{proof}

\begin{proof}[Proof of Proposition~\ref{prop:7new}]
  By Malric's density theorem of zeros, recalled in \eqref{eq:density},
  $\gamma_n^*\to 1^{-}$ almost surely.  
  Hence 
  it is enough to show that on the event 
  $\event{Y>0}\cap\event{\gamma_n^*\to 1^{-}}$ the set 
  $\tilde{H}=[0,\infty)\setminus\tilde{A}(C,s)$ is porous at 1. 

  Let $\xi=Y/2$ and
  \begin{displaymath}
    I_n=(\gamma^*_n,\gamma^*_n+r_n)
    ,\quad\text{where}\quad r_n=\zfrac{\xi\wedge C}{C}^2(1-\gamma^*_n). 
  \end{displaymath}
  We claim that if 
  \begin{equation}
    \label{eq:gamma abs bn}
    \xi>0,\quad\gamma_n=\gamma_n^*>s,\quad\text{and}\quad
    \abs*{\bn[k]_{\gamma_n}}>\xi\sqrt{1-\gamma_n},\quad\text{for $0\leq k<n$}.   
  \end{equation}
  then $I_n\subset \tilde{A}(C,s)\cap( \gamma_n^*,1)$ with
  $r_n/(1-\gamma^*_n)>0$ not  depending on $n$. Since on 
  $\event{Y>0}\cap\event{\gamma_n^*\to 1^{-}}$ the condition \eqref{eq:gamma
    abs bn} holds for infinitely many $n$, we obtain the porosity at 1. 

  So assume that \eqref{eq:gamma abs bn} holds 
  for $n$ at a given $\omega$. As $I_n\subset(\gamma_n^*,1)$, for $t\in
  I_n$ we have that $s<t<1$ and
  $st<s<\gamma_n(t)=\gamma_n^*(t)=\gamma_n=\gamma_n^*$, 
  that is, the first 
  requirement in \eqref{tilde ACs}: $st<\gamma_n(t)=\gamma_n^*(t)$ holds for
  any $t\in I_n$. 
  For the other  requirement, note that
  $t-\gamma_n(t)< r_n\leq (1-\gamma_n^*)\xi^2/C^2$ yields 
  \begin{displaymath}
    \min_{0\leq k<n}\abs*{\bn[k]_{\gamma_n}}>\xi\sqrt{1-\gamma^*_n}>
    C\sqrt{t-\gamma_n(t)},\quad\text{for $t\in I_n$}.\qedhere
  \end{displaymath}
\end{proof}

\subsection{Proof of Theorem~\ref{prop:8}}
\label{sec:3.5}

  Compared to Theorem~\ref{prop:7} in the proof of Theorem~\ref{prop:8}   we
  consider an even larger 
  set $[0,\infty)\setminus \dtilde{A}(C,s)$, where
  \begin{multline*}
    \dtilde{A}(C,s)=\set{t>0}{
      \exists n\geq1,\, s t<\gamma_n(t)=\gamma_n^*(t),\, \\
      \min_{0\leq k<n} \abs*{\bn[k]_{\gamma_n(t)}}>C\sqrt{t-\gamma_n(t)},\\
      \max_{u\in[\gamma_n(t),t]}\abs*{\beta_u-\beta_{\gamma_n(t)}}
      < 
      \sqrt{t-\gamma_n(t)}
    }\subset
    \tilde{A}(C,s)\subset A(C,s). 
  \end{multline*}
  Here we also require that  the fluctuation of
  $\beta$ between $\gamma_n(t)$ and $t$ is not too big. 

  We will prove the next proposition below.
  \begin{proposition}\label{prop:100}
    Let $C>1$, and $s\in(0,1)$. Then  almost surely on the event
    $\event{X<1}$, the closed set $[0,\infty)\setminus\dtilde{A} (C,s)$ is
    porous at 1.   
  \end{proposition}
  
  This result implies that if $X<1$ almost surely, then for any $C>0$,
  $s\in(0,1)$ the random closed set
  $[0,\infty)\setminus\dtilde{A} (C,s)$ is porous at 1 
  almost surely, and so is the smaller set $[0,\infty)\setminus A(C,s)$. Then
  the strong mixing of $\lT$ follows from Corollary~\ref{cor:8}.

  To complete the proof of Theorem~\ref{prop:8}, it remains to show that $X=0$
  almost surely on the event $\event{X<1}$. This is the content of
  next proposition. In order to prove Theorem~\ref{thm:XY} we introduce a new
  parameter $L>0$. 
  \begin{multline*}
    \dtilde{A}_L(C,s)=\set{t>0}{
      \exists n\geq1,\, s t<\gamma_n(t)=\gamma_n^*(t),\,\\ 
      \min_{0\leq k<n} \abs*{\bn[k]_{\gamma_n(t)}}>C\sqrt{t-\gamma_n(t)},\,
      \max_{u\in[\gamma_n(t),t]}\abs*{\beta_u-\beta_{\gamma_n(t)}}
      <
      L\sqrt{t-\gamma_n(t)}
    }
  \end{multline*}
  Then $\dtilde{A}(C,s)=\dtilde{A}_1(C,s)$.
  \begin{proposition}\label{prop:X<1} 
    Fix $L\geq 1$ and set
    \begin{displaymath}
      \dtilde{A}_L(s)=\bigcap_{C>0} \dtilde{A}_L(C,s),
      \quad\text{for $s\in(0,1)$ and}\quad
      \dtilde{A}_L=\bigcap_{s\in(0,1)} \dtilde{A}_L(s).
    \end{displaymath}
    Then the events $\event{X=0}$, $\event{X<1}$, $\event*{1\in\dtilde{A}_L}$
    and $\event*{1\in\dtilde{A}_L(s)}$,  
    $s\in(0,1)$ are equal up to null sets. 
  \end{proposition}

\begin{proof}[Proof of Proposition~\ref{prop:X<1}]
    Fix $s\in(0,1)$ $L\geq1$ and let $C>L$. 
    Assume that $1\in \dtilde{A}_L(C,s)$. Let $n>0$ be an index which
    witnesses the containment. Then, as $C>L$ we can apply Lemma~\ref{l1} to
    see that the  absolute increments of $\bn[0],\dots,\bn[n]$ between
    $\gamma_n$ and $1$  are the same.
    This implies that 
    \begin{displaymath}
      \abs*{\bn[k]_1}\geq
      \abs*{\bn[k]_{\gamma_n}}-\abs*{\bn[k]_1-\bn[k]_{\gamma_n}}=
      \abs*{\bn[k]_{\gamma_n}}-\abs*{\beta_1-\beta_{\gamma_n}}, \quad\text{for
        $0\leq k\leq n$},
    \end{displaymath}
    hence
    \begin{displaymath}
      Z_n\geq \min_{0\leq
        k<n}\abs*{\bn[k]_{\gamma_n}}-\abs*{\beta_1-\beta_{\gamma_n}}>
    C\sqrt{1-\gamma_n}-L\sqrt{1-\gamma_n}
    \end{displaymath}
    whereas
    \begin{displaymath}
      Z_{n+1}\leq
      \abs*{\bn_1}=\abs*{\bn_1-\bn_{\gamma_n}}=
      \abs*{\beta_1-\beta_{\gamma_n}}
      <L\sqrt{1-\gamma_n}. 
    \end{displaymath}
    Thus
    \begin{displaymath}
      \inf_{n\geq 0} \frac{Z_{n+1}}{Z_n}\leq \frac{L}{C-L},\quad 
      \text{on $\event{1\in\dtilde{A}_L(C,s)}$ almost surely}, 
    \end{displaymath}
    and
    \begin{equation}\label{eq:frac small}
      \inf_{n\geq 0}\frac{Z_{n+1}}{Z_n}=0 
      ,\quad 
      \text{on 
        $\event{1\in\dtilde{A}_L(s)}%
        $ almost surely}.
    \end{equation}
    But, $Z_{n+1}/Z_{n}>0$ almost surely for all $n$, 
    hence $X=\liminf_{n\to\infty} Z_{n+1}/Z_n=0$ almost surely on
    $\event*{1\in\dtilde{A}_L(s)}$. This proves $\event*{1\in
      \dtilde{A}_L(s)}\subset\event{X=1}$. 

    Next, the law of the random closed set $[0,\infty)\setminus\dtilde{A}_L(C,s)$
    is clearly invariant under scaling, hence by Proposition
    \ref{prop:100} and Lemma~\ref{l7}
    \begin{equation}\label{eq:p>P{X<1}}
      \event{X<1} \subset 
      \event{\text{$[0,\infty)\setminus\dtilde{A}_L(C,s)$ is  porous at 1}}
      =\event{1\in \dtilde{A}_L(C,s)},
    \end{equation}
    each relation holding up to a null set.

    The inclusion  $\dtilde{A}_L(C',s')\subset\dtilde{A}_L(C,s)$ for
    $C'\geq C>0$ and 
    $0<s\leq s'<1$ yields
    \begin{displaymath}
      \dtilde{A}_L=\bigcap_{k=1}^\infty \dtilde{A}_L(k,1-1/k).
    \end{displaymath}
    Hence,
    $\event*{X<1}\subset\event*{1\in\dtilde{A}_L}\subset\event*{1\in\dtilde{A}_L(s)}$ 
    almost surely, which 
    together with $\event*{1\in\dtilde{A}_L(s)}\subset\event{X=0}$ completes
    the proof.  
\end{proof}

To prove Proposition \ref{prop:100} we need a Corollary of the Blumenthal $0-1$ law.

\begin{corollary}\label{cor:14}%
  Let $(x_n)$ be a sequence of non-zero numbers tending to zero, $\P$ the
  Wiener measure on $\cC[0,\infty)$
  and $D\subset \cC[0,\infty)$ be a Borel set such that
  $\P{D}>0$.

  Then $\P{\Theta_{x_n}^{-1}( D)\io}=1$.
\end{corollary}
\begin{proof}
  Recall that the canonical
  process on $\cC[0,\infty)$ was denoted by $\beta$. We also use the notation 
  $\B_t=\sigma\set{\beta_s}{s\leq t}$.

  We approximate $D$ with $D_n\in \B_{t_n}$ such that $\sum \P{D\triangle
    D_n}<\infty$, where $\triangle$ denotes the symmetric difference
  operator. Passing to a subsequence if necessary, we may assume that  
  $t_nx_n^2\to0$. Then, since $\Theta_{x_n}^{-1} (D_n)\in\B_{t_nx_n^2}$, we have
  that $\event{\Theta_{x_n}^{-1}(D_n),\io}\in\cap_{s>0}\B_s$, and
  the Blumenthal $0-1$ law ensures  that
  $\P{\Theta_{x_n}^{-1}(D_n),\io}\in\smallset{0,1}$. 

  But $\sum \P{\Theta_{x_n}^{-1}(D)\triangle \Theta_{x_n}^{-1}(D_n)}<\infty$ since
  $\Theta_{x_n}$ preserves $\P$. Borel--Cantelli lemma shows that, almost
  surely, $\Theta_{x_n}^{-1}(D)\triangle \Theta_{x_n}^{-1}(D_n)$ occurs for finitely
  many $n$. Hence $\P{\Theta_{x_n}^{-1}(D),\io}\in\smallset{0,1}$.

  Fatou lemma applied to the indicator functions of $\Theta_{x_n}^{-1} (D)^c$
  yields
  \begin{displaymath}
    \P{\Theta_{x_n}^{-1}(D),\io}\geq
    \limsup_{n\to\infty}\P{\Theta_{x_n}^{-1}(D)}=\P{D}>0.
  \end{displaymath}
  Hence $\P{\Theta_{x_n}^{-1}(D),\io}=1$.
\end{proof}

\begin{proof}[Proof of Proposition~\ref{prop:100}]
  We work on the event $\event{X<1}$.
  Set $\xi=(1/X -1)/2$. 
  Then $1<\xi+1  <1/X$
  and  
  \begin{displaymath}
    1< 1+\xi <\limsup_{n\to\infty} \frac{Z_n}{Z_{n+1}}=\limsup_{n\to\infty}
    \frac{Z_n}{\abs*{\bn_1}}. 
  \end{displaymath}
  Hence
  \begin{displaymath}
    \min_{0\leq k<n} \abs*{\bn[k]_1}=Z_n> (1+\xi) \abs*{\bn_1},\quad\text{for
      infinitely many $n$}.
  \end{displaymath}
  Let $n_1<n_2<\ldots$ the enumeration of those indices, and set
  $x_k=\hn[n_k](1)\bn[n_k]_1$ for $k\geq 1$. The inequality
  $\abs*{\bn[n_k]_1}<(1+\xi)^{-1} \abs*{\bn[n_{k-1}]_1}$ shows that $x_k\to0$.

  Call $B$ the Brownian motion defined by $B_t=\beta_{t+1}-\beta_1$ and for
  real numbers $\delta,C>0$  set
  \begin{multline*}
    D(\delta,C)=\set{w\in \cC[0,\infty)}{
      \text{$\sup_{t\leq 2}\abs{w(t)}< 1+\delta$;}\\
      \text{$w+1$ has a zero in $[0,1]$, but no zero in $(1,2]$;
      }\\
      \max_{t\in[\gamma,2]}\abs{w(t)+1}\leq\frac{\delta\wedge C}{2C} 
      \text{, where $\gamma$ is the last zero of $w+1$ in $[0,2]$} 
    }.
  \end{multline*}
  For each $\delta,C>0$ the Wiener measure puts positive, although possibly
  very small, probability on $D(\delta,C)$. 
  Then Corollary~\ref{cor:14} yields that the Brownian motion $B$ takes values
  in the random sets $\Theta_{x_k}^{-1} D(\xi,C)$ for infinitely many $k$ on
  $\event{\xi>0}=\event{X<1}$ almost
  surely; since the random variables $x_k$, $\xi$ are $\B_1$-measurable, and $B$
  is independent of $\B_1$.  

  \smallskip
  For $k\geq 1$ let $\tgamma_k=\gamma_{n_k}(1+x_k^2)$, that is, the
  last zero of $\bn[n_k]$ before $1+x_k^2$ and set
  \begin{displaymath}
    I_k=(\tgamma_k +\tfrac12r_k,\tgamma_k+r_k),\quad\text{where}\quad
    r_k=\zfrac{\xi\wedge C}{C}^2 x_k^2. 
  \end{displaymath}
  This interval is similar to the one used in the proof of
  Proposition~\ref{prop:7new}, but now we use only the  right half of the
  interval $(\tgamma_k,\tgamma_k+r_k)$.  
  
  Next we show that 
  \begin{equation}
    \label{eq:B}
    B\in\Theta_{x_k}^{-1} D(\xi,C),\quad\text{and}\quad 
    s\leq (1+x_k^2)^{-1}   
  \end{equation}
  implies 
  \begin{equation}
    \label{eq:Ik}
    I_k\subset \dtilde{A}(C,s)\cap (1,1+2x_k^2).
  \end{equation}
  By definition
  $r_k/(4x_k^2)$, the ratio of the lengths of $I_k$ and $(1,1+2x_k^2)$, does
  not depend on $k$.   Then the  porosity of $[0,\infty)\setminus
  \dtilde{A}(C,s)$ at 1 follows 
  for almost every point of $\event{X<1}$, as we have seen that 
  \eqref{eq:B} holds for infinitely many $k$ almost surely on $\event{X<1}$.

  So assume that \eqref{eq:B} holds for $k$  at a given $\omega$. 
  The key observations are that then 
  \begin{align}\label{eq:beta B}
    \bn[\ell]_{1+t}-\bn[\ell]_1&=
    \hn[\ell]_1 B_t
    ,
    &&\text{for $0\leq \ell\leq n_k$,  
      $0\leq t\leq 2x_k^2$},\\
    \label{eq:gamma star}
    \gamma_\ell(t)&<1,
    &&\text{for $0\leq \ell<n_k$ and $1\leq t\leq
      1+2x_k^2$},\\
    \label{eq:gamma n_k}
    \gamma_{n_k}(t)&=\tgamma_k>1,
    &&\text{for $t\in[\tgamma_k,1+2x_k^2]$}.
  \end{align}
  First, we prove \eqref{eq:beta B}--\eqref{eq:gamma n_k} and then with their
  help we derive  $I_k\subset \dtilde{A}(C,s)$.

  To get \eqref{eq:beta B} and \eqref{eq:gamma star} we apply Lemma \ref{l1}
  to  $I=[1,1+2x_k^2]$, $n=n_k$ and $a=1$. This can be done since we
  have
  \begin{align}\label{choice of nk}
    \min_{0\leq \ell<n_k}\abs*{\bn[\ell]_1} &>(1+\xi) \abs*{x_k},\quad\text{by the
      choice of $n_k$},\\
    \label{sup beta}
    \max_{t\in[1,1+2x_k^2]} \abs*{\beta_t-\beta_1} &< (1+\xi)\abs*{x_k},\quad
    \text{since $\Theta_{x_k} B\in D(\xi,C)$ by \eqref{eq:B}}.
  \end{align}
  (\ref{l1:it1}) of Lemma \ref{l1} is exactly \eqref{eq:gamma star}, while
  (\ref{l1:it2}) 
  of the same Lemma gives \eqref{eq:beta B} if we note that
  $B_t=\beta_{1+t}-\beta_1$ by definition.

  \eqref{eq:gamma n_k} claims two things: $\bn[n_k]$ has a zero 
  in $(1,1+x_k^2]$, but has no zero in $(1+x_k^2,1+2x_k^2]$.  Write
  \eqref{eq:beta B} with $\ell=n_k$: 
  \begin{equation*}\label{eq:bnk}
    \bn[n_k]_{1+t}=\bn[n_k]_{1}+\hn[n_k](1) B_t=\hn[n_k](1)(x_k+B_t),
    \quad\text{for $0\leq t\leq 2x_k^2$}. 
  \end{equation*}
  Next, we use that $\Theta_{x_k} B\in D(\xi,C)$, whence $1+\Theta_{x_k}B $
  has a zero in $[0,1]$ 
  but no zero in $(1,2]$. Then the relation
  \begin{equation}\label{eq:Theta B}
    x_k\br{1+(\Theta_{x_k} B)_v}=x_k+%
    B_{x_k^2 v}=
    \hn[n_k](1) %
    \bn[n_k]_{1+x_k^2 v}
  \end{equation}
  justifies \eqref{eq:gamma n_k}.

  \bigskip

  To finish the proof,   it remains to show that  $I_k\subset
  \dtilde{A}(C,s)$, since by \eqref{eq:gamma n_k} $\tgamma_k$ the last zero of
  $\bn[n_k]$ before $1+x_k^2$ is greater than 1, so $I_k\subset (1,1+2x_k^2)$
  holds.  

  Fix $t\in I_k$.  We need    to check the next three properties.
  \begin{nlist}
  \item  \textit{$st<\gamma_{n_k}(t)=\gamma_{n_k}^*(t)$.}
    
    \medskip
    
    By \eqref{eq:gamma n_k} $\gamma_{n_k}(t)=\tgamma_k>1$ and by the
    definition of $I_k$ we have 
    $1<\tgamma_k<t<\tgamma_k+r_k\leq \tgamma_k+x_k^2$. 
    Hence, 
    \begin{displaymath}
      \gamma_{n_k}(t)= \tgamma_k > \frac{\tgamma_k}{\tgamma_k+x_k^2}t
      > \frac{1}{1+x_k^2} t\geq   st,
    \end{displaymath}
    where we used $s\leq(1+x_k^2)^{-1}$, the second part of \eqref{eq:B}. 
    
    By \eqref{eq:gamma star}, $\gamma_{n_k}(t)=\gamma_{n_k}^*(t)$, as  $t\in
    I_k\subset[1,1+2x_k^2]$. 

    \bigskip

  \item \textit{$\displaystyle\min_{0\leq \ell<n_k}
      \abs*{\bn[\ell]_{\tgamma_k}}>    C\sqrt{t-\tgamma_k}$.}

    \medskip
    
    Since $x_k=\hn[n_k](1)\bn[n_k]_1$,  
    $\bn[n_k]_{\tgamma_k}=0$ and $\tgamma_k\in[1,1+x_k^2]$, 
    \eqref{eq:beta B} yields 
    \begin{displaymath}
      \max_{0\leq\ell<n_k}\abs*{\bn[\ell]_{\tgamma_k}-\bn[\ell]_1}=
      \abs*{\bn[n_k]_{\tgamma_k}-\bn[n_k]_1}=\abs*{\bn[n_k]_1}=\abs*{x_k}.
    \end{displaymath}
    Then, by the triangle inequality and \eqref{choice of nk} 
    \begin{align*}
      \min_{0\leq\ell<n_k}\abs*{\bn[\ell]_{\tgamma_k}}&\geq
      \min_{0\leq\ell<n_k}\abs*{\bn[\ell]_1}-
      \max_{0\leq\ell<n_k}\abs*{\bn[\ell]_{\tgamma_k}-\bn[\ell]_1}
      \\ &>
      (1+\xi) \abs*{x_k}-\abs*{x_k} 
      = \xi \abs*{x_k}.
    \end{align*}
    On the other hand $\sqrt{t-\tgamma_k}<\sqrt{r_k}\leq \abs*{x_k} \xi/C$,
    hence 
    \begin{displaymath}
      \min_{0\leq\ell<n_k}\abs*{\bn[\ell]_{\tgamma_k}}>
      \xi \abs*{x_k}\geq C\sqrt{t-\tgamma_{k}}.
    \end{displaymath}

    \smallskip

  \item \textit{$\displaystyle\max_{u\in [\tgamma_k,t]} 
      \abs*{\beta_u-\beta_{\tgamma_k}}<
      \sqrt{t-\tgamma_k}$. 
    }

    \medskip

    $1+\Theta_{x_k} B$ has  a zero in
    $[0,1]$ but no zero in $(1,2]$, since $\Theta_{x_k} B\in D(\xi,C)$. Denote
    as above by $\gamma$ its last zero in 
    $[0,1]$. Then by relation \eqref{eq:Theta B} we have that
    $\tgamma_k=1+x_k^2\gamma$ and 
    \begin{displaymath}
      \max_{u\in[\tgamma_k,1+2x_k^2]} \abs*{\bn[n_k]_u}= \abs{x_k}
      \max_{v\in[\gamma,2]}\abs*{1+(\Theta_{x_k}B)_v}\leq \abs{x_k}\frac{\xi\wedge
        C}{2C}= \frac{\sqrt{r_k}}2.
    \end{displaymath}
    Writing \eqref{eq:beta B} with $\ell=n_k$ and using that
    $\bn[n_k]_{\tgamma_k}=0$ and $t<1+2x_k^2$ we obtain 
    \begin{displaymath}
      \max_{u\in[\tgamma_k,t]}\abs*{\beta_u-\beta_{\tgamma_k}}=
      \max_{u\in[\tgamma_k,t]}\abs*{\bn[n_k]_u}\leq 
      \max_{u\in[\tgamma_k,1+2x_k^2]}\abs*{\bn[n_k]_u}\leq \frac{\sqrt{r_k}}2.
    \end{displaymath}

    By the definition of $I_k$ we have $t-\tgamma_k > \tfrac12 r_k$.
    Hence 
    \begin{displaymath}
      \max_{u\in [\tgamma_k,t]}\abs*{\beta_u-\beta_{\tgamma_k}}
      \leq \frac{\sqrt{r_k}}{2}<
      \frac{\sqrt{r_k}}{2}\sqrt{\frac{t-\tgamma_k}{\tfrac12 r_k}}<
      \sqrt{t-\tgamma_k}.\qedhere
    \end{displaymath}  
  \end{nlist}
\end{proof}

\subsection{Proof of Theorem~\ref{thm:XY}}
\label{sec:3.6}
In this subsection we prove the equality of the events
$\event{X=0}$, $\event{Y=\infty}$ and
$\event{1\in A}$ up to null sets, where
\begin{displaymath}
  A=\bigcap_{s\in(0,1)}A(s),\quad\text{with}\quad A(s)=\bigcap_{C>0} A(C,s).
\end{displaymath}
We keep the notation introduced in Propositions~\ref{prop:Y>0} and
\ref{prop:X<1} for $\dtilde{A}_L(s)$, $\dtilde{A}_L$ and $\tilde{A}$.

\smallskip 

Recall that  $\dtilde{A}_L\subset\tilde{A}\subset{A}$ by definition for any
$L\geq1$.  Then by Propositions~\ref{prop:Y>0} and \ref{prop:X<1} we have  
\begin{equation}
  \label{eq:chain}
  \event{X=0}=\event*{1\in\dtilde{A}_L}
  \subset\event*{1\in\tilde{A}}=
  \event{Y=\infty}\subset \event{1\in A}.
\end{equation}

For $C>0$ let 
\begin{equation*}
  \tau_C=\inf\set{t\geq \tfrac12}{\exists n\geq0, \,\bn_t=0\,\min_{0\leq
      k<n}\abs*{\bn[k]_t}\geq C\sqrt{(1-t)\vee0}}.
\end{equation*}
We show below that 
\begin{equation}\label{eq:chain1}
  \event{1\in A}\subset
  \bigcap_{C>0}\event{\tau_C<1},\quad\text{up to null a set},
\end{equation}
and
\begin{equation}
  \label{eq:P=}
  \P{\bigcap_{C>0}\event{\tau_C<1}}\leq\P{X=0}. 
\end{equation}
Then the claim follows by concatenating \eqref{eq:chain} and
\eqref{eq:chain1}, and observing that the largest and the smallest events in
the obtained chain of almost inclusions has the same probability by
\eqref{eq:P=}.  

We start with \eqref{eq:chain1}. 
If $1\in A$ then $1\in A(C,s)$ for every $s\in(0,1)$, especially
for $s_0=\gamma_0\vee 1/2$, where $\gamma_0$ is the last zero of
$\beta$ before 1, we have $1\in A(C,s_0)$. 
Then, by the definition of $A(C,s_0)$ there is an integer $n\geq0$
and a real number $\gamma\in(s_0,1)$ such that $\bn_\gamma=0$ and $\min_{0\leq
  k<n} \abs*{\bn[k]_1}>C\sqrt{1-\gamma}$. The integer $n$ cannot be zero since
$\bn[0]=\beta$ has no zero in $(s_0,1)$. Thus $\tau_C\leq \gamma<1$, which
shows  the inclusion.

Next, we turn to \eqref{eq:chain1}. Fix $C>L\geq1$ and let
\begin{displaymath}
 \gamma=\sup\set{s\in[\tau_C,1]}{\beta_s=\beta_{\tau_C}}. 
\end{displaymath}
Let us show that 
\begin{equation}\label{eq:A(C,L)}
  \event{\tau_C<1\text{ and }%
    \max_{\tau_C\leq t\leq1}\abs*{\beta_t-\beta_{\tau_C}}< %
    L\sqrt{1-\gamma}}
  \subset
  \event{1\in \dtilde{A}_L(C,\tfrac12)}.
\end{equation}
Indeed, on the event on the left hand side of \eqref{eq:A(C,L)} 
there exists  a random index $n$ such that $\bn_{\tau_C}=0$ and
\begin{displaymath}
  \min_{0\leq k\leq n-1}\abs*{\bn[k]_{\tau_C}}\geq 
  C\sqrt{1-\tau_C}> L\sqrt{1-\gamma} > 
  \max_{{\tau_C}\leq t\leq 1}\abs*{\beta_t-\beta_{\tau_C}}.  
\end{displaymath}
Then we can apply Lemma~\ref{l1} with $I=[\tau_C,1]$, $a=\tau_C$ and
$n=n$. We obtain that 
$\bn[k]$ has no zero in $[\tau_C,1]$ for $k=0,\dots,n-1$, and 
the absolute  increments
$\abs*{\bn[k]_t-\bn[k]_{\tau_C}}$,  are the same
for $k=0,\dots,n$
 and $t\in[\tau_C,1]$.
In particular, $\bn[k]_\gamma=\bn[k]_{\tau_C}$ for every $0\leq k\leq n$, 
$\gamma$ is the last zero of $\bn$ in $[\tau_C,1]$ and
$\gamma=\gamma_n=\gamma_n^*$. 
Moreover,
\begin{displaymath}
  \min_{0\leq k< n}\abs*{\bn[k]_{\gamma_n^*}}=\min_{0\leq k< n}\abs*{\bn[k]_{\tau_C}}
  \geq C\sqrt{1-\tau_C}>C\sqrt{1-\gamma_n^*}.
\end{displaymath}
So  $n$ and $\gamma_n^*$ witnesses that $1\in\dtilde{A}_L(C,\frac12)$, since 
we also have that 
\begin{displaymath}
  \max_{t\in[\gamma_n^*,1]}\abs*{\beta_t-\beta_{\gamma_n^*}}\leq
  \max_{t\in[\tau_C,1]}\abs*{\beta_t-\beta_{\tau_C}}< L\sqrt{1-\gamma_n^*}.
\end{displaymath}

\medskip

From \eqref{eq:A(C,L)}, by the strong Markov property and the scaling
invariance of $\beta$, we obtain
\begin{displaymath}
  \P{\tau_C<1}\times  
  \P{\max_{t\in[0,1]}\abs*{\beta_t}\leq L\sqrt{1-\gamma_0}}\leq  
  \P{1\in\dtilde{A}_L(C,\tfrac12)}.
\end{displaymath}
Letting $C$  go to infinity and using Proposition~\ref{prop:X<1}, this yields
\begin{align*}
  \P{\bigcap_{C>0} \event{\tau_C<1}}\times 
    \P{\max_{t\in[0,1]}\abs*{\beta_t}\leq L\sqrt{1-\gamma_0}}&\leq 
  \P{1\in\dtilde{A}_L(\tfrac12)} %
  \\
  &=\P{X=0}. %
\end{align*}
This is true for all $L\geq 1$. Thus \eqref{eq:P=} is obtained by letting $L$ go
to infinity.

\subsection{Proof of Theorem \ref{thm:tightness}} 

\label{sec:3.7}

In this subsection we prove that the tightness of $\set{x\nu(x)}{x\in(0,1)}$
and $\set{nZ_n}{n\geq1}$ are equivalent and both implies $X<1$ almost surely.

\medskip

Fix $K>0$. By definition $\event{(K/n)\nu(K/n)> K}=\event{nZ_{n}\geq K}$ for
any $n\geq1$. For small $x>0$
values there is $n$ such that $K/n<x<2K/n$ and $x\nu(x)\leq (2K/n)\nu(K/n)$
by the monotonicity of $\nu$. But, then
$\event{x\nu(x)>2K}\subset\event{nZ_{n}>K}$. Hence
\begin{displaymath}
  \limsup_{x\to 0^+} \P{x\nu(x)> 2K}\leq \limsup_{n\to\infty}\P{nZ_n\geq K}\leq 
  \limsup_{x\to0^+} \P{x\nu(x)> K}.
\end{displaymath}
So the tightness of the two families are equivalent and it is enough to prove
that when $\set{x\nu(x)}{x\in(0,1)}$ is tight then $X<1$ almost surely.

We have the next easy lemma,
whose proof is sketched at the end of this subsection. %
\begin{lemma}\label{l12}
  \begin{displaymath}
    X=\liminf_{n\to\infty}\frac{Z_{n+1}}{Z_n} = 
    \liminf_{x\to0^+}%
    \frac{\abs*{\bn[\nu(x)]_1}}{x}.    
  \end{displaymath}
\end{lemma}

Then we have that 
\begin{displaymath}
  \I{X> 1-\delta}\leq \liminf_{x\to 0+} \I{\abs*{\bn[\nu(x)]_1}/x> 1-\delta}. 
\end{displaymath}
Hence, by Fatou lemma
\begin{align*}
  \P{X>1-\delta}\leq\liminf_{x\to0+}\P{\abs*{\bn[\nu(x)]_1}>x(1-\delta)}.
\end{align*}
Let $x\in(0,1)$ and $K>0$. Since on the event 
\begin{displaymath}
  \event{\nu(x)\leq \frac Kx}\cap \event{\abs*{\bn[\nu(x)]_1}>x(1-\delta)}
\end{displaymath}
at least one of the standard
normal variables $\bn[k]_1$, $0\leq k\leq K/x$ takes values in a set of size
$2x\delta$, namely in $(-x,-x(1-\delta))\cup (x(1-\delta),x)$, 
\begin{align*}
  \P{\frac{\abs*{\bn[\nu(x)]_1}}x>1-\delta}  \hspace{-2em}&\\&\leq
  \P{\nu(x)>\frac K{x}}+
  \zjel{\frac{K}{x}+1}\P{1-\delta<\frac{\abs{\beta_1}}{x}<1}\\ &\leq 
  \P{x\nu(x)>K}+(K+1)\delta.
\end{align*}
In the last step we used that the standard normal density is bounded by
$1/\sqrt{2\pi}$, whence $\P{1-\delta<\frac{\abs{\beta_1}}{x}<1}\leq \delta x$. 

By the tightness assumption for any $\eps>0$ there exists  $K_\eps$ such that
$\sup_{x\in(0,1)}\P{x\nu(x)>K_\eps}\leq \eps$. 
Hence, 
\begin{displaymath}
  \P{X=1}=\lim_{\delta\to0+} \P{X>1-\delta}\leq 
  \lim_{\delta\to0+}\eps+ (K_\eps+1) \delta=\eps.
\end{displaymath}
Since, this is true for all $\eps>0$, we get that $\P{X=1}=0$ and the proof of
Theorem \ref{thm:tightness} is complete.\qed

\begin{proof}[Proof of Lemma~\ref{l12}] Since
  $Z_{\nu(x)}=\abs*{\bn[\nu(x)]_1}$ Lemma \ref{l12} 
  is a particular case of the following claim: if $(a_n)$ is a decreasing
  sequence of positive numbers tending to zero then  
  \begin{displaymath}
    \liminf_{k\to\infty}\frac{a_{k+1}}{a_k}=\liminf_{x\to0^+}\frac{a_{n(x)}}x,
  \end{displaymath}
  where $n(x)=\inf\set{k\geq 1}{a_k<x}$. First, for  $x<a_1$ the relation 
  $a_{n(x)-1}\geq x>a_{n(x)}$  gives 
  \begin{displaymath}
    \frac{a_{n(x)}}{a_{n(x)-1}}\leq \frac{a_{n(x)}}x
  \end{displaymath}
  and
  \begin{displaymath}
    \liminf_{k\to\infty}\frac{a_{k+1}}{a_k}\leq \liminf_{x\to0^+}\frac{a_{n(x)}}x.  
  \end{displaymath}

  For the opposite direction, for every $k\geq 0$, $a_{n(a_k)}<a_k$, therefore
  $a_{n(a_k)}\leq a_{k+1}$ as $(a_n)$ is non-increasing. Since $a_k\to0$ as
  $k\to\infty$, one gets
  \begin{displaymath}
    \liminf_{x\to 0^+} \frac{a_{n(x)}}x\leq 
    \liminf_{k\to\infty} \frac{a_{n(a_k)}}{a_k}\leq 
    \liminf_{k\to\infty} \frac{a_{k+1}}{a_k}.\qedhere
  \end{displaymath}
\end{proof}

\begin{acknowledgement}
  Crucial part of this work was done while visiting the University of
  Strasbourg, in February of 2011. I am very grateful to IRMA and especially to
  professor Michel \'Emery for their invitation and for their hospitality.

  Conversations with Christophe Leuridan and Jean Brossard, and a few days later
  with Marc Yor and Marc Malric inspired the first formulation of
  Theorem~\ref{prop:7} and \ref{prop:8}.  

  The author is grateful to the anonymous referee for his detailed reports and
  suggestions that improved the presentation of the results
  significantly. Especially, the referee proposed a simpler argument for
  stronger result in Theorem~\ref{thm:tightness}, pointed out a sloppiness in
  the proof of Theorem~\ref{prop:8}, suggested a better
  formulation of Theorems~\ref{prop:7} and \ref{prop:8} and one of his/her
  remarks motivated Theorem~\ref{thm:XY}. 

  The original proof of Corollary \ref{cor:14} was based on the
  Erd\H{o}s--R\'enyi generalization of the Borel--Cantelli lemma, see
  \cite{MR787923}. The somewhat shorter proof in the text was proposed by the
  referee.  
\end{acknowledgement}

\def\doi#1{doi:~\href{http://dx.doi.org/#1}{\nolinkurl{#1}}}
\def\eprint#1{arXiv:~\href{http://arxiv.org/abs/#1}{\nolinkurl{#1}}}


\begin{thebibliography}{12}
\providecommand{\natexlab}[1]{#1}
\providecommand{\url}[1]{\texttt{#1}}
\providecommand{\urlprefix}{URL }
\expandafter\ifx\csname urlstyle\endcsname\relax
  \providecommand{\doi}[1]{doi:\discretionary{}{}{}#1}\else
  \providecommand{\doi}{doi:\discretionary{}{}{}\begingroup
  \urlstyle{rm}\Url}\fi
\providecommand{\eprint}[2][]{\url{#2}}

\bibitem[{\textsc{Billingsley}(1968)}]{MR0233396}
Billingsley, P.: Convergence of probability measures.
\newblock John Wiley \& Sons Inc., New York (1968)

\bibitem[{\textsc{Brossard and Leuridan}(2012)}]{Leuridan2009}
Brossard, J., Leuridan, C.: Densit\'e des orbites des trajectoires browniennes
  sous l'action de la transformation de {L}\'evy.
\newblock Ann. Inst. H. Poincar\'e Probab. Statist.  (2012).
\newblock \textbf{48}(2), 477--517.
\newblock \doi{10.1214/11-AIHP463}

\bibitem[{\textsc{Dubins and Smorodinsky}(1992)}]{MR1231991}
Dubins, L.E., Smorodinsky, M.: The modified, discrete, {L}\'evy-transformation
  is {B}ernoulli.
\newblock In: S\'eminaire de {P}robabilit\'es, {XXVI}, \emph{Lecture Notes in
  Math.}, vol. 1526, pp. 157--161. Springer, Berlin (1992)

\bibitem[{\textsc{Dubins et~al.}(1993)\textsc{Dubins, {\'E}mery, and
  Yor}}]{MR1308559}
Dubins, L.E., {\'E}mery, M., Yor, M.: On the {L}\'evy transformation of
  {B}rownian motions and continuous martingales.
\newblock In: S\'eminaire de {P}robabilit\'es, {XXVII}, \emph{Lecture Notes in
  Math.}, vol. 1557, pp. 122--132. Springer, Berlin (1993)

\bibitem[{\textsc{Fujita}(2008)}]{MR2417970}
Fujita, T.: A random walk analogue of {L}\'evy's theorem.
\newblock Studia Sci. Math. Hungar.  (2008).
\newblock \textbf{45}(2), 223--233.
\newblock \doi{10.1556/SScMath.45.2008.2.50}

\bibitem[{\textsc{Jacod and Shiryaev}(1987)}]{MR959133}
Jacod, J., Shiryaev, A.N.: Limit theorems for stochastic processes,
  \emph{Grundlehren der Mathematischen Wissenschaften [Fundamental Principles
  of Mathematical Sciences]}, vol. 288.
\newblock Springer-Verlag, Berlin (1987)

\bibitem[{\textsc{Malric}(1995)}]{MR1318194}
Malric, M.: Transformation de {L}\'evy et z\'eros du mouvement brownien.
\newblock Probab. Theory Related Fields  (1995).
\newblock \textbf{101}(2), 227--236.
\newblock \doi{10.1007/BF01375826}

\bibitem[{\textsc{Malric}(2003)}]{MR1975087}
Malric, M.: Densit\'e des z\'eros des transform\'es de {L}\'evy it\'er\'es d'un
  mouvement brownien.
\newblock C. R. Math. Acad. Sci. Paris  (2003).
\newblock \textbf{336}(6), 499--504

\bibitem[{\textsc{Malric}(2010)}]{malric-2010}
Malric, M.: Density of paths of iterated {L}\'evy transforms of {B}rownian
  motion.
\newblock ESAIM: Probability and Statistics  (2010).
\newblock \doi{10.1051/ps/2010020}

\bibitem[{\textsc{M{\'o}ri and Sz{\'e}kely}(1983)}]{MR787923}
M{\'o}ri, T.F., Sz{\'e}kely, G.J.: On the {E}rd{\H o}s-{R}\'enyi generalization
  of the {B}orel-{C}antelli lemma.
\newblock Studia Sci. Math. Hungar.  (1983).
\newblock \textbf{18}(2-4), 173--182

\bibitem[{\textsc{Revuz and Yor}(1991)}]{revuz-yor}
Revuz, D., Yor, M.: Continuous martingales and {B}rownian motion,
  \emph{Grundlehren der Mathematischen Wissenschaften [Fundamental Principles
  of Mathematical Sciences]}, vol. 293.
\newblock Springer-Verlag, Berlin (1991)

\bibitem[{\textsc{Riesz and Sz.-Nagy}(1955)}]{MR0071727}
Riesz, F., Sz.-Nagy, B.: Functional analysis.
\newblock Frederick Ungar Publishing Co., New York (1955).
\newblock Translated by Leo F. Boron

\end{thebibliography}
\end{document}